\documentclass[a4paper,12pt,reqno]{amsart}
\usepackage{amssymb}
\usepackage{amsmath}
\usepackage{amsthm}
\usepackage{amsfonts}
\usepackage{mathtools}
\usepackage{enumitem}
\makeatletter
\def\dicho#1{\expandafter\@dicho\csname c@#1\endcsname}
\def\@dicho#1{\ifnum#1>1 or \else\fi(\@Roman#1)}
\makeatother
\AddEnumerateCounter{\dicho}{\@dicho}{or (III)}
\newlist{dichotomy}{enumerate}{1}
\setlist[dichotomy]{label=\dicho*,leftmargin=1.5cm}

\usepackage{mathrsfs}
\usepackage[all]{xy}
\usepackage{hyperref}
\usepackage{cleveref}
\usepackage{url}
\usepackage{comment}
\usepackage{xcolor}
\usepackage{pifont}

\usepackage[british]{babel}
\usepackage[margin=1.2in]{geometry}
\numberwithin{equation}{section} \numberwithin{figure}{section}

\usepackage{makecell}
\usepackage{multirow}
\usepackage{etoolbox}
\usepackage[title]{appendix}
\usepackage{eqparbox}
\newcommand{\eqmathbox}[2][M]{\eqmakebox[#1]{$\displaystyle#2$}}
\usepackage{pbox}

\newcommand{\fp}{\mathfrak{p}}
\newcommand{\fq}{\mathfrak{q}}

\newcommand{\NotTypeOneTwoPrimes}{\textup{\textsf{NotTypeOneTwoPrimes}}}
\newcommand{\TypeTwoPrimes}{\textup{\textsf{TypeTwoPrimes}}}
\newcommand{\TypeOnePrimes}{\textup{\textsf{TypeOnePrimes}}}
\newcommand{\TypeTwoNotMomosePrimes}{\textup{\textsf{TypeTwoNotMomosePrimes}}}
\newcommand{\TypeTwoMomosePrimes}{\textup{\textsf{TypeTwoMomosePrimes}}}
\newcommand{\PrimesUpTo}{\textup{\textsf{PrimesUpTo}}}

\newcommand{\F}{\mathbb{F}}
\newcommand{\eps}{\varepsilon}
\newcommand{\PP}{\mathbb{P}}
\newcommand\ZZ{\mathbb{Z}}

\newcommand{\Q}{\mathbb{Q}}

\newcommand{\Nm}{\textup{Nm}}
\newcommand{\Mod}[1]{\ (\mathrm{mod}\ #1)}
\newcommand{\id}{\textup{id}}
\newcommand{\Aux}{\textup{\textsf{Aux}}}
\newcommand{\Gen}{\textup{\textsf{Gen}}}
\newcommand{\AuxGen}{\textup{\textsf{AuxGen}}}

\renewcommand{\O}{\mathcal{O}}
\DeclareMathOperator{\IsogPrimeDeg}{\textup{\textsf{IsogPrimeDeg}}}
\DeclareMathOperator{\IsogCyclicDeg}{\textup{\textsf{IsogCyclicDeg}}}
\DeclareMathOperator{\DLMV}{\textup{\textsf{DLMV}}}

\DeclareMathOperator{\Aut}{Aut}
\DeclareMathOperator{\lcm}{lcm} 
\DeclareMathOperator{\Spec}{Spec}
\DeclareMathOperator{\Gal}{Gal}
\DeclareMathOperator{\ord}{ord}
\DeclareMathOperator{\Cl}{\textup{Cl}}
\DeclareMathOperator{\Supp}{\textup{\textsf{Supp}}}

\theoremstyle{case}

\newtheorem{lemma}{Lemma}

\newtheorem{theorem}[lemma]{Theorem}
\newtheorem*{claim}{Claim}
\newtheorem{algorithm}[lemma]{Algorithm}
\newtheorem{proposition}[lemma]{Proposition}
\newtheorem{corollary}[lemma]{Corollary}
\numberwithin{table}{section}

\theoremstyle{definition}
\newtheorem{example}[lemma]{Example}

\newtheorem{definition}[lemma]{Definition}
\newtheorem{remark}[lemma]{Remark}

\newtheorem*{condC}{Condition C}
\newtheorem*{condCC}{Condition CC}
\newtheorem*{ack}{Acknowledgements}

\numberwithin{lemma}{section}

\AtBeginEnvironment{appendices}{\crefalias{section}{appendix}}

\begin{document}

\title[Quadratic Isogeny Primes]
{Explicit isogenies of prime degree over quadratic fields}

\author{\sc Barinder S. Banwait}
\address{Barinder S. Banwait \\
Harish-Chandra Research Institute\\
A CI of Homi Bhabha National Institute\\
Chhatnag Road\\
Jhunsi\\
Prayagraj - 211019\\
India}
\curraddr{
Fakult\"{a}t f\"{u}r Mathematik\\
Universit\"{a}t Heidelberg\\
Im Neuenheimer Feld 205\\
69120 Heidelberg\\
Germany}
\email{barinder.s.banwait@gmail.com}


\begin{abstract}
Let $K$ be a quadratic field which is not an imaginary quadratic field of class number one. We describe an algorithm to compute the primes $p$ for which there exists an elliptic curve over $K$ admitting a $K$-rational $p$-isogeny. This builds on work of David, Larson-Vaintrob, and Momose. Combining this algorithm with work of Bruin-Najman, \"{O}zman-Siksek, and most recently Box, we determine the above set of primes for the three quadratic fields $\Q(\sqrt{-10})$, $\Q(\sqrt{5})$, and $\Q(\sqrt{7})$, providing the first such examples after Mazur's 1978 determination for $K = \Q$. The termination of the algorithm relies on the Generalised Riemann Hypothesis.
\end{abstract}

\maketitle

\section{Introduction}

Let $K$ be a number field, and $N$ a positive integer.

\begin{definition}
We say that $N$ is a \textbf{cyclic isogeny degree for $K$} if there exists an elliptic curve $E/K$ which possesses a $K$-rational cyclic isogeny of degree $N$. We denote the set of such integers, for a given $K$, by $\IsogCyclicDeg(K)$. If $N$ is prime, we refer to $N$ as an \textbf{isogeny prime for $K$}, and denote the set of such primes by $\IsogPrimeDeg(K)$.
\end{definition}

Consider the base case of $K = \Q$. By an approach which has come to be known as \emph{Mazur's formal immersion method} - involving a delicate study of the N\'{e}ron model of the Eisenstein quotient of the Jacobian of the modular curve $X_0(N)$ - Mazur proved the following.

\begin{theorem}[Mazur, Theorem 1 in \cite{mazur1978rational}]
\[ \IsogPrimeDeg(\Q) = \left\{2, 3, 5, 7, 11, 13, 17, 19, 37, 43, 67, 163\right\}.\]
\end{theorem}

It is worth stressing that, prior to this theorem, it was not known that the set $\IsogPrimeDeg(\Q)$ was even finite.

Moreover, in the introduction to his paper, Mazur explains how consideration of the ``graph of rational isogenies'' reduces the determination of the larger set $\IsogCyclicDeg(\Q)$ to the explicit determination of $X_0(N)(\Q)$ for a handful of composite values $N$, which for all but five values of $N$ - $13^2, 13 \cdot 7, 13 \cdot 5, 13 \cdot 3$ and $5^3$ - had already been carried out. It was then subsequently Kenku, in a series of four papers \cite{kenku1}, \cite{kenku3}, \cite{kenku2}, \cite{kenku4}, who established that $X_0(N)(\Q)$ consists only of the cuspidal points for these five values of $N$, thereby yielding the complete resolution of cyclic isogeny degrees over the rationals.

\begin{theorem}[Kenku]
\[ \IsogCyclicDeg(\Q) = \left\{1 \leq N \leq 19\right\} \cup \left\{21, 25, 27, 37, 43, 67, 163\right\}.\]
\end{theorem}

These values of cyclic isogeny degrees for $\Q$ were known prior to the work of Mazur, and can be explained in terms of clearly describable geometric conditions on the modular curve $X_0(N)$; see the opening paragraph and table of \cite{mazur1978rational} for more details.

For an arbitrary number field $K$, it is no longer necessarily true that $\IsogPrimeDeg(K)$ (and hence \emph{a fortiori} $\IsogCyclicDeg(K)$) is finite. To see this, consider an elliptic curve $E$ which has complex multiplication by an order $\mathcal{O}$ of an imaginary quadratic field $L$. Observe that, for any rational prime $p$ which splits or ramifies in $\mathcal{O}$, $E$ admits an endomorphism of degree $p$ rational over the ring class field $H_\mathcal{O}$ of $\mathcal{O}$, and so defines a \textbf{CM point} in $X_0(p)(H_\mathcal{O})$. Thus, if $K$ contains the Hilbert class field of an imaginary quadratic field, then $\IsogPrimeDeg(K)$ is infinite.

Let us therefore consider a number field $K$ which does not contain the Hilbert class field of an imaginary quadratic field. In attempting to establish finiteness of $\IsogPrimeDeg(K)$ in this case, Momose \cite{momose1995isogenies} carried out a systematic study of the \emph{mod-$p$ isogeny character $\lambda$} arising from considering the Galois action on the kernel of a $p$-isogeny, and showed (Theorem A in \emph{loc. cit.}) the existence of a constant $C_K$ such that, if $p > C_K$, then $\lambda$ must fall into one of three types: ``Type $1$'', ``Type $2$'', or ``Type $3$''. We will recall this theorem (\Cref{thm:momose_isogeny}) in \Cref{sec:prelims} below; for now we remark that, by definition of ``Type $3$'', it does not arise under our hypothesis on $K$, so we exclude it from consideration for the remainder of the Introduction.

Momose showed that Type $1$ primes arise only finitely often under the additional hypothesis that $[K : \Q] \leq 12$ (Theorem $3$ in \emph{loc. cit.}); that Type $2$ primes arise only finitely often under the Generalised Riemann Hypothesis (Remark $8$ in \emph{loc. cit.}); and that Type $2$ primes arise only finitely often - unconditional on GRH but without an effective bound - if $K$ is a quadratic field (Proposition $1$ in \emph{loc. cit.}). In this way, Momose was able to establish the following unconditional result.

\begin{theorem}[Momose, Theorem B in \cite{momose1995isogenies}]
Let $K$ be a quadratic field which is not an imaginary quadratic field of class number $1$. Then $\IsogPrimeDeg(K)$ is finite.
\end{theorem}

To be sure - and to quote Mazur from the introduction of \cite{Mazur3} speaking in the context of rational points on $X_0(N)$ - the assertion of mere finiteness is not all that is wanted. The work of making Momose's constant $C_K$ above explicit was subsequently and independently carried out unconditionally by David (Th\'{e}or\`{e}me I', Proposition 2.4.1 in \cite{david2008caractere}; see also Th\'{e}or\`{e}me II in \cite{david2012caractere}), and conditionally on GRH by Larson and Vaintrob (Theorem 7.9 in \cite{larson_vaintrob_2014}; see also Remark 7.8 which explains the extent to which assuming GRH yields a better bound on $C_K$). Both David as well as Larson and Vaintrob make the observation that, as a consequence of Merel's resolution of the uniform boundedness conjecture for torsion orders of elliptic curves over number fields \cite{merel1996bornes}, Type 1 primes as above arise only finitely often; one therefore obtains the following result, written explicitly in the literature by the latter authors.

\begin{theorem}[Larson and Vaintrob, Corollary 2 in \cite{larson_vaintrob_2014}\label{thm:isog-finite}]
Assume the Generalised Riemann Hypothesis. For a number field $K$, $\IsogPrimeDeg(K)$ is finite if and only if $K$ does not contain the Hilbert class field of an imaginary quadratic field.
\end{theorem}

An explicit expression for $C_K$ does not immediately yield an explicit upper bound on $\IsogPrimeDeg(K)$; one further needs to effectively bound Type 2 primes, which Theorem 7.9 in \cite{larson_vaintrob_2014} does only up to the determination (depending on GRH) of various effectively computable absolute constants. Since these have not been computed, an explicit upper bound on $\IsogPrimeDeg(K)$ has not hitherto been exhibited for any $K \neq \Q$. 

Specialise now to $[K : \Q] = 2$. By combining the bounds of David with those of Larson and Vaintrob, and improving upon the process of determining an upper bound for Type 2 primes, we obtain the following. 

\begin{algorithm}
Assume GRH. Then there is an algorithm which, given a quadratic field $K$ which is not imaginary quadratic of class number $1$, computes an upper bound on $\IsogPrimeDeg(K)$.
\end{algorithm}

Given the provenance of this bound, we henceforth refer to it as the \textbf{David-Larson-Momose-Vaintrob} bound for $K$, or $\DLMV(K)$ for short. The algorithm has been implemented in Sage \cite{sagemath}, and running it on the quadratic fields $\Q(\sqrt{D})$ (not imaginary quadratic of class number one) with $|D| \leq 10$ yields the first explicit upper bounds on $\IsogPrimeDeg(K)$ since Mazur's work in the 70s.

\begin{corollary}\label{cor:DLMV}
Assume GRH. Then for each $K$ in \Cref{tab:dlmv}, $\DLMV(K)$ gives a bound on the isogeny primes for $K$.
\end{corollary}

\begin{table}[htp]
\begin{center}
\begin{tabular}{|c|c|c|}
\hline
$\Delta_K$ & $K$ & $\DLMV(K)$\\
\hline
$-40$ & $\Q(\sqrt{-10})$ & $3.20 \times 10^{316}$\\
$-24$ & $\Q(\sqrt{-6})$ & $2.99 \times 10^{308}$\\
$-20$ & $\Q(\sqrt{-5})$ & $2.58 \times 10^{305}$\\
$8$ & $\Q(\sqrt{2})$ & $4.06 \times 10^{139}$\\
$12$ & $\Q(\sqrt{3})$ & $1.68 \times 10^{152}$\\
$5$ & $\Q(\sqrt{5})$ & $5.65 \times 10^{126}$\\
$24$ & $\Q(\sqrt{6})$ & $9.76 \times 10^{177}$\\
$28$ & $\Q(\sqrt{7})$ & $1.08 \times 10^{189}$\\
$40$ & $\Q(\sqrt{10})$ & $2.59 \times 10^{354}$\\
\hline
\end{tabular}
\vspace{0.3cm}
\caption{\label{tab:dlmv}The David-Larson-Momose-Vaintrob bounds for quadratic fields $\Q(\sqrt{D})$ with $|D| \leq 10$, excluding imaginary quadratic fields of class number one.}
\end{center}
\end{table}

According to the introduction of \cite{BPR}, $10^{80}$ is approximately the number of atoms in the visible universe. The adjective `astronomical' therefore falls somewhat short of describing the numbers in \Cref{tab:dlmv}. Indeed, as observed by Kamienny and Mazur in the introduction of \cite{kamienny1995rational}, such enormous bounds, however explicit, are of small comfort when one expects $\IsogPrimeDeg(K)$ to be not much larger than $\IsogPrimeDeg(\Q)$, and that in parading such bounds one risks some level of public embarrassment\footnote{especially so for us considering our bounds are several orders of magnitude larger than theirs!}.

We are thus motivated to go further. Rather than merely enumerating an upper bound for $\IsogPrimeDeg(K)$, we devise a novel method, independent of the $\DLMV$ bound, to show that the isogeny primes essentially all arise in the support of a handful of explicitly computable integers, which may be considered as the main essential contribution of the present paper. The following is a summarised version of the full version (\Cref{thm:main_full}) to be found in \Cref{sec:parts_together}.

\begin{theorem}\label{alg:main}
Assume GRH. Then there is an algorithm which, given a quadratic field $K$ of discriminant $\Delta_K$ which is not imaginary quadratic of class number $1$, computes a superset for $\IsogPrimeDeg(K)$ as the union of four explicitly computable sets:
\begin{align*}
    \IsogPrimeDeg(K) \subseteq  \NotTypeOneTwoPrimes(K) &\cup \TypeOnePrimes(K) \\
        \cup \ \TypeTwoPrimes(K) &\cup \Supp(\Delta_K).
\end{align*}
\end{theorem}

More specifically, while the algorithm to compute the sets $\NotTypeOneTwoPrimes(K)$ and $\TypeOnePrimes(K)$ utilises a multiplicative sieve to obtain fairly tight supersets and is implemented in Sage, the algorithm to compute $\TypeTwoPrimes(K)$ requires one to check whether a certain elementary condition on Legendre symbols is satisfied for all prime numbers up to the bound on Type 2 primes. While this bound is fairly sizeable (about $60$ billion for $K = \Q(\sqrt{5})$), it is nowhere near the dizzying heights of $\DLMV(K)$, and more importantly it can be managed in short order with an optimised and parallel-threaded PARI/GP \cite{PARI} implementation. Note however that, while the finiteness of this bound on Type 2 primes is unconditional in this case, its effective determination \emph{is} conditional on GRH, and therefore GRH is fundamentally required to obtain an algorithm which provably terminates. See the discussion at the end of \Cref{sec:type_2_primes} for more on the dependence on GRH.

The implementation of the algorithm is therefore split between PARI/GP and Sage. As will be explained in \Cref{sec:type_1_primes}, the primes $p \leq 71$ will necessarily always be contained in the output; and clearly $\IsogPrimeDeg(\Q) \subseteq \IsogPrimeDeg(K)$ for all $K$; so in the following table we show only those primes $p \geq 73$ that furthermore are not contained in $\IsogPrimeDeg(\Q)$.

\begin{corollary}\label{cor:DLMV}
Assume GRH. Then for each $K$ in \Cref{tab:put_isog_primes}, if $p$ is an isogeny prime for $K$ which is strictly larger than $71$ and not equal to $163$, then $p$ is listed in the column \textup{Large Possible Isogeny Primes}.
\end{corollary}

\begin{table}[htp]
\begin{center}
\begin{tabular}{|c|c|c|}
\hline
$\Delta_K$ & $K$ & Large Possible Isogeny Primes\\
\hline
$-40$ & $\Q(\sqrt{-10})$ & 73\\
$-24$ & $\Q(\sqrt{-6})$ & 73, 97, 103\\
$-20$ & $\Q(\sqrt{-5})$ & 73, 89, 97\\
$8$ & $\Q(\sqrt{2})$ & 79\\
$12$ & $\Q(\sqrt{3})$ & 97\\
$5$ & $\Q(\sqrt{5})$ & 73, 79\\
$24$ & $\Q(\sqrt{6})$ & 73, 97\\
$28$ & $\Q(\sqrt{7})$ & 73\\
$40$ & $\Q(\sqrt{10})$ & 73, 79, 97\\
\hline
\end{tabular}
\vspace{0.3cm}
\caption{\label{tab:put_isog_primes}The large possible isogeny primes ($p \geq 73$ and $p \notin \IsogPrimeDeg(\Q)$) for the same quadratic fields as in \Cref{tab:dlmv}.}
\end{center}
\end{table}

Having obtained a manageable list of possible isogeny primes for the $K$s in the above table, we should now like to determine, for each prime $p$ in this list, whether or not it is indeed an isogeny prime; that is, whether or not $X_0(p)(K)$ contains any noncuspidal points.

Such questions are notoriously difficult; but fortunately, many of these $p$'s we need to consider are such that the genus of $X_0(p)$ is fairly small; and most importantly, there has been in recent years much progress in the study of \emph{quadratic points on low-genus modular curves}. Of particular significance are the works of Bruin and Najman \cite{bruin2015hyperelliptic}, \"{O}zman and Siksek \cite{ozman2019quadratic}, and most recently Box \cite{box2021quadratic}; taken together, these three works give a complete determination of the quadratic points on $X_0(N)$ when it has genus $2$, $3$, $4$ or $5$. In essence, for each such $N$, there are only finitely many quadratic points which do not correspond to elliptic $\Q$-curves; and these finitely many points are determined explicitly. All of these works utilise Magma \cite{magma} in a significant way.

By combining these results with earlier work of \"{O}zman in determining local solubility of quadratic twists of $X_0(N)$ \cite{ozman2012points}, together with a fair amount of Sage and Magma computation - notably the Chabauty package - one is able to completely determine $\IsogPrimeDeg(K)$ for some $K$s, giving - conditional upon GRH - the first examples of the determination of isogeny primes for number fields larger than $\Q$.

\begin{theorem} \label{thm:main}
Assuming GRH, we have the following.
\begin{align*}
\IsogPrimeDeg(\Q(\sqrt{-10})) &= \IsogPrimeDeg(\Q)\\
\IsogPrimeDeg(\Q(\sqrt{5})) &= \IsogPrimeDeg(\Q) \cup \left\{23, 47\right\}\\
\IsogPrimeDeg(\Q(\sqrt{7})) &= \IsogPrimeDeg(\Q).
\end{align*}
\end{theorem}

Of particular note here is the ruling out of the prime $73$, which is treated in Box's work, and relies on the determination of the rational points on the modular curve $X_0(73)^+$, recently completed in the work of Balakrishnan, Best, Bianchi, Lawrence, M\"{u}ller, Triantafillou and Vonk (Theorem 6.3 in \cite{balakrishnan2019two}) and uses the method of explicit quadratic Chabauty developed by Balakrishnan and Dogra \cite{balakrishnan2018quadratic} and Balakrishnan, Dogra, M\"{u}ller, Tuitman and Vonk \cite{balakrishnan2019explicit} following Kim's work \cite{kim2005motivic} \cite{kim2009unipotent} on non-abelian Chabauty. The arising of the primes $23$ and $47$ may be explained by the geometric condition that, in these cases, $X_0(p)$ is hyperelliptic, so necessarily admits infinitely many quadratic points; it is therefore not too surprising that these primes arise for a given quadratic field.

Our original motivation for this work was to write down $\IsogPrimeDeg(K)$ for a \emph{single $K \neq \Q$}, so we are content to stop at this point. Dealing with primes larger than $73$ may be possible by extending the methods of \"{O}zman-Siksek and Box, but currently this is the largest prime considered in their work. However, in \Cref{appendix}, written jointly with Derickx, it is proved that $79 \notin \IsogPrimeDeg(\Q(\sqrt{5}))$; this is used in the case of $\Q(\sqrt{5})$ above.

Our implementation \cite{isogeny_primes} is available at:

\begin{center}
\url{https://github.com/barinderbanwait/quadratic_isogeny_primes}
\end{center}

More details on how to use the command line tool are given on the \verb|README.md| there. In addition, the repository also contains a \emph{User's guide} which explains in more detail the Sage implementations of the algorithms, explaining the various methods and helper functions we use, as well as the PARI/GP optimisation of searching for Type 2 primes. The Magma code we use for the determination of isogeny primes up to $73$ is also presented there.

The outline of this paper is as follows. In \Cref{sec:prelims} we give an overview of Momose's isogeny types and David's more careful subsequent treatment of Momose's ideas in the case of a number field $K$ Galois over $\Q$. David's work in making explicit one of Momose's constants is also shown in this section. \Cref{sec:overview} then specialises to quadratic fields, and describes our algorithm for dealing with isogeny primes which are \emph{not} of Type 1 or Type 2. \Cref{sec:type_1_primes,sec:type_2_primes} then deal respectively with Type 1 and Type 2 primes, and \Cref{sec:parts_together} combines the results of the previous sections to present the full version of \Cref{alg:main} as well as the different parts of the $\DLMV$ bound. \Cref{thm:main} is proved in the final \Cref{sec:weeding}.

In the spirit of Mazur's introduction in \cite{mazur1978rational}, we would like to end our introduction with the following.\\

\emph{A question.} Suppose one has determined $\IsogPrimeDeg(K)$ for a particular quadratic field $K$. By considering the ``graph of $K$-rational isogenies'', can one reduce the determination of $\IsogCyclicDeg(K)$ to the determination of $K$-points on $X_0(N)$ for a finite list of composite $N$s, just as Mazur does in the introduction of his paper?

\ack{I am very grateful to David Roe and Andrew Sutherland for granting me access to MIT's \verb|legendre| computer, on which the Magma and PARI/GP computations were performed, and to Barry Mazur for a correspondence which prompted me to think beyond the computations into the underlying geometric reasons for isogenies. The copy of Magma was made available through a generous initiative of the Simons Foundation. I am indebted to Maarten Derickx for numerous suggestions on improving the implementation as well as for ideas that led to the Appendix; Andreas Enge, Aurel Page and Michael Stoll for implementation improvements; Samir Siksek for answering a question about rational points on hyperelliptic curves; James Stankewicz for bringing my attention to Momose's work; Isabel Vogt for useful conversations; and the two anonymous referees for their very careful and conscientious reading of previous versions of the manuscript and for their numerous comments and suggestions on them which not only improved the exposition, but also highlighted gaps in my previous implementation (specifically in the subtle difference between Type 2 and Momose Type 2 primes). I also thank Alex Bartel, Lassina Demb\'{e}l\'{e}, Bas Edixhoven, Gerhard Frey, Philippe Michaud-Jacobs, and John Voight for comments on an earlier version of the manuscript.

Finally I wish to express my sincere gratitude to John Cremona, who introduced me to the subject of rational isogenies, and who both encouraged and helped me in my first steps in attempting to understand the 1978 paper of Mazur, whose work more generally remains a great source of inspiration to me.
}

\section{Recap on isogeny types}\label{sec:prelims}

This section collects existing results from the literature to be used in the sequel concerning the isogeny character, as well as the identification of certain types - Type 1, Type 2, and Type 3 - of such characters. In particular, we will discuss Theorem 1/Theorem A of \cite{momose1995isogenies}, hereafter referred to as \emph{Momose's Isogeny Theorem}, in the case that $K$ is Galois over $\Q$. We largely follow David's paper \cite{david2012caractere}, whose exposition includes the reproving in greater detail of many of Momose's original results \cite{momose1995isogenies}. The reader may also wish to consult Section 5 of \cite{mazur1978rational} as well as David's PhD thesis \cite{david2008caractere}.

We thus take a number field $K$ which is Galois over $\Q$, and let $E/K$ be an elliptic curve over a number field possessing a $K$-rational $p$-isogeny for some prime $p \geq 17$, assumed unramified in $K$. Write $W$ for the kernel of the isogeny, a one-dimensional $G_K$-stable module, and denote by $\lambda$ the \emph{isogeny character}:
\[ \lambda : G_K \longrightarrow \Aut W(\overline{K}) \cong \F_p^\times.\]
Since the character $\lambda^{12}$ plays such a pivotal role in the proceedings, we denote it by $\mu$. The following result summarises the essential facts about $\mu$.

\begin{proposition}[David]\label{prop:david_isog_char}
Let $E/K$ be an elliptic curve over a number field possessing a $K$-rational $p$-isogeny for some prime $p \geq 17$, assumed unramified in $K$. Let $\lambda$ denote the isogeny character, and write $\mu = \lambda^{12}$. Write $G_K = \Gal(\overline{K}/K)$, and denote by $\chi_p$ the mod-$p$ cyclotomic character of $G_K$. Then we have the following.
\begin{enumerate}
	\item
	$\mu$ is unramified outside the primes of $K$ above $p$.
	\item
	Let $\fp$ be a prime of $K$ above $p$, and write $I_\fp$ for the corresponding inertia subgroup of $G_K$. Then there exists a unique integer $a_\fp$ valued in the set $\left\{0,4,6,8,12\right\}$ such that $\mu$ restricted to $I_\fp$ is equal to $\chi_p^{a_\fp}$.
	\item
	If $a_\fp$ = $4$, $6$ or $8$, then $E$ has potentially good reduction at $\fp$. Write $\widetilde{E}_\fp$ for the reduction of $E$ at $\fp$ over $\overline{\F_p}$.
	\item
	If $a_\fp$ = $4$ or $8$, then $\Aut(\widetilde{E}_\fp) \cong \mu_6$, $p \equiv 2 \Mod{3}$, and $j(E) \equiv 0 \Mod{\fp}$.
	\item
	If $a_\fp$ = $6$, then $\Aut(\widetilde{E}_\fp) \cong \mu_4$, $p \equiv 3 \Mod{4}$, and $j(E) \equiv 1728 \Mod{\fp}$.
\end{enumerate}
\end{proposition}

\begin{proof}
(1) is explained as Propositions 1.4 and 1.5 in \cite{david2012caractere} (see also Propositions 3.3 and 3.5 in \cite{david2011borne}). (2) is Propositions 1.2 and 1.3 in \cite{david2012caractere} (see also Proposition 3.2 in \cite{david2011borne}). We briefly sketch some of the details.

This result is obtained by building on the precise description of the action of $I_\fp$ on the $p$-torsion Galois module $E[p]$ found in Section 1 of \cite{SerreGal}. Having dealt with the case of $E$ having potentially multiplicative reduction at $\fp$ (in which case one can conclude that $a_\fp$ must be either $0$ or $12$, which in particular gives (3)), one obtains a minimal extension $K_\fp'$ of the local field $K_\fp$ over which $E$ has good reduction, and one denotes by $e_\fp$ the ramification degree of this extension. As explained in the proof of Lemme 2.4 of \cite{david2011borne} (whose explanation goes further back to the proof of Theorem 2 of \cite{serretate}) the integer $e_\fp$ must divide the order of the automorphism group $\Aut(\widetilde{E}_\fp)$ (over $\overline{\F_p}$) of the reduction of $E/K_\fp'$, and thus (since $p$ is assumed to be greater than $13$ and hence not equal to $2$ or $3$) one obtains that $e_\fp$ must be in the set $\left\{1,2,3,4,6\right\}$, and moreover:

\begin{itemize}
	\item
	if $e_\fp$ = $3$ or $6$, then $\Aut(\widetilde{E}_\fp) \cong \mu_6$ and $j(E) \equiv 0 \Mod{\fp}$;
	\item
	if $e_\fp$ = $4$, then $\Aut(\widetilde{E}_\fp) \cong \mu_4$ and $j(E) \equiv 1728 \Mod{\fp}$.
\end{itemize}

As explained in the proof of Proposition 3.2 of \cite{david2011borne}, the integer $a_\fp$ is then obtained as $(12 / e_\fp) \cdot r_\fp$ for $r_\fp$ an integer between $0$ and $e_\fp$ which moreover must satisfy a certain congruence condition (labelled $\star$ in \emph{loc. cit.}). Running through the possible values one concludes that $a_\fp$ must lie in the set $\left\{0,4,6,8,12\right\}$, and moreover (see the summary table at the end of page 2184 of \emph{loc. cit.}) on obtains the extra conditions on $p$ and $\Aut(\widetilde{E}_\fp)$ in the $a_\fp$ = $4$, $6$ and $8$ cases expressed in (4) and (5).
\end{proof}

\begin{remark}
Similar conditions on $p$ and $j(E)$ arise in the description of the bad fibres of the N\'{e}ron model of $J_0(N)$ in the Appendix of \cite{Mazur3}, by Mazur and Rapoport. Moreover, strictly speaking, since the integers $a_\fp$ give the exponent of $\chi_p$ of $\mu$ restricted to $I_\fp$, they are only defined modulo $p - 1$. However, since $p \geq 17$, we may take them to be actual integers.
\end{remark}

Write $G := \Gal(K/\Q)$. We fix a prime $\fp_0$ of $K$ above $p$; this choice allows us to define, for $\tau \in G$, $a_\tau$ to be the integer $a_\fp$ corresponding to the ideal $\fp = \tau^{-1}(\fp_0)$. Taken together, these integers $(a_\tau)_{\tau \in G}$ describe how $\mu$, when identified (via class field theory) with a character on the group $I_K(p)$ of ideals of $K$ coprime to $p$, acts on principal ideals $(\alpha)$ of $K$ coprime to $p$; this is expressed as Lemma 1 of \cite{momose1995isogenies}. However, we will follow David and \emph{not} identify $\lambda$ in this way (that is, for us, $\lambda$ will remain a Galois character). David's version of Momose's Lemma 1 is then Proposition 2.6 of \cite{david2012caractere} (see also Corollaire 2.2.2 of \cite{david2008caractere}), expressed via a \emph{twisted norm element}, which David writes as
\[ \mathcal{N}(\alpha) = \prod_{\tau \in G}\tau(\alpha)^{a_\tau};\]
however we prefer to use Momose's approach of packaging the $(a_\tau)_{\tau \in G}$ into an element $\eps$ of the group ring $\ZZ[G]$ acting on $K^\times$, viz. $\eps = \sum_{\tau \in G}a_\tau\tau$. Following Freitas and Siksek in \cite{freitas2015criteria} (see the discussion just before Proposition 2.2 in \emph{loc. cit.}) we henceforth refer to $\eps$ as the \textbf{signature of the isogeny character $\lambda$}, or sometimes simply the \textbf{isogeny signature}. David's twisted norm $\mathcal{N}(\alpha)$ is then expressed as $\alpha^\eps$ (the action of $\eps$ on $\alpha$), which makes explicit the dependence on the integers $a_\tau$. The aforementioned result is then expressed as follows, where $\iota_{\fp_0}$ denotes the inclusion of $K$ into the completion $K_{\fp_0}$, and $\sigma_\fq$ denotes the Frobenius automorphism at $\fq$.

\begin{lemma}[Momose, Lemma~1 of \cite{momose1995isogenies}, Proposition 2.6 of \cite{david2012caractere}]\label{lem:momose_lemma_1}
Assume that $K$ is Galois over $\Q$, and that $p$ is unramified in $K$. Let $\alpha \in K^\times$ coprime to $p$, and let $\prod_{\fq \nmid p}\fq^{\ord_\fq(\alpha)}$ denote the prime ideal decomposition of the principal ideal $\alpha\mathcal{O}_K$. Then one has
\[ \prod_{\fq \nmid p}\mu(\sigma_\fq)^{\ord_\fq(\alpha)} = \iota_{\fp_0}(\alpha^\eps) \Mod{\fp_0}.\]
\end{lemma}

Note that if one of the integers in the signature $\eps$ is $4$ or $8$ (respectively $6$), then from the above discussion of the integers $a_\fp$, we have that $\Aut(\widetilde{E}_\fp) \cong \mu_6$ and $p \equiv 2 \Mod{3}$ (respectively $\Aut(\widetilde{E}_\fp) \cong \mu_4$ and $p \equiv 3 \Mod{4}$). We therefore refer to such signatures as \textbf{sextic} (respectively \textbf{quartic}) by reference to the size of $\Aut(\widetilde{E}_\fp)$. The overlap of these two designations (i.e. the signature contains both $4$ and $6$ or both $8$ and $6$) is an interesting source of isogenies, and will be referred to as \textbf{mixed}. Note also that if $p$ is inert or totally ramified in $K$, then there is only one prime ideal of $K$ above $p$, and thus all of the integers in the signature must be the same; we call such as signature \textbf{constant}. We observe (for later recollection) that a different choice of $\fp_0$ merely permutes the integers in the signature, and that we may consider signatures - \emph{a priori} of the existence of any $p$-isogeny - merely as elements of the set $\left\{0,4,6,8,12\right\}^G$. It is in fact this last observation which allows us to bound prime-degree-isogenies having an isogeny character of signature $\eps$, and when we want to stress that we are considering $\eps$ in this way, we will merely consider it as an element of $\left\{0,4,6,8,12\right\}^G$, or say that $\eps$ is a \emph{possible} isogeny signature.

Let $\fq$ be a prime ideal of $K$ which is coprime to $p$, and consider the reduction of $E$ at $\fq$, which is either potentially multiplicative, potentially good supersingular, or potentially good ordinary. In both of the potentially good cases, the characteristic polynomial of Frobenius $\sigma_\fq$ acting on the $p$-adic Tate module of $E$ has coefficients in $\ZZ$ and is independent of $p$ (Theorem 3 in \cite{serretate}); we may thus write $P_\fq(X)$ for this polynomial. This is a quadratic polynomial whose roots have absolute value $\sqrt{\Nm(\fq)}$. We write $L^\fq$ for the splitting field of this polynomial, which is either $\Q$ or an imaginary quadratic field.

In all cases of reduction at $\fq$, one obtains congruence conditions modulo $p$ on $\mu(\sigma_\fq)$, which subsequently yield divisibility conditions on $p$. In order to express these we make the following definition.

\begin{definition}\label{def:ABC}
Let $K/\Q$ be a Galois number field, $\fq$ a prime ideal of $K$ of order $h_\fq$ in the class group of $K$, and $\gamma_\fq$ a generator of the principal ideal $\fq^{h_\fq}$. Then we define the following integers.
\begin{align*}
A(\eps, \fq) &:= \Nm_{K/\Q}(\gamma_\fq^\eps - 1)\\
B(\eps, \fq) &:= \Nm_{K/\Q}(\gamma_\fq^\eps - \Nm(\fq)^{12h_\fq})\\
C_s(\eps, \fq) &:= \lcm(\left\{ \Nm_{K(\beta)/\Q}(\gamma_\fq^\eps - \beta^{12h_\fq}) \ | \ \beta \mbox{ is a supersingular Frobenius root over }\F_\fq\right\})\\
C_o(\eps, \fq) &:= \lcm(\left\{ \Nm_{K(\beta)/\Q}(\gamma_\fq^\eps - \beta^{12h_\fq}) \ | \ \beta \mbox{ is an ordinary Frobenius root over }\F_\fq\right\})\\
C(\eps, \fq) &:= \lcm\left(C_o(\eps, \fq), C_s(\eps, \fq)\right)\\
ABC(\varepsilon, \fq) &:= \lcm(A(\varepsilon, \fq), B(\varepsilon, \fq), C(\varepsilon, \fq), \Nm(\fq)),
\end{align*}
where in $C_s(\eps, \fq)$ (respectively $C_o(\eps, \fq)$)  the $\lcm$ is taken over all roots $\beta$ of characteristic polynomials of Frobenius of supersingular (respectively, ordinary) elliptic curves defined over the residue field $\F_\fq$ of $K$ at $\fq$.
\end{definition}

The divisibility conditions are then expressed as follows.

\begin{proposition}\label{prop:ABC_div}
Let $E/K$ be an elliptic curve over a number field $K$ which admits a $K$-rational $p$-isogeny of isogeny signature $\eps$. Let $\fq$ be a prime of $K$ coprime to $p$.
\begin{enumerate}
	\item
	If $E$ has potentially multiplicative reduction at $\fq$, then $p$ divides $A(\varepsilon, \fq)$ or $B(\varepsilon, \fq)$.
	\item
	If $E$ has potentially good ordinary (respectively, supersingular) reduction at $\fq$, then $p$ divides $C_o(\varepsilon, \fq)$ (respectively, $C_s(\varepsilon, \fq)$).
\end{enumerate}
\end{proposition}

\begin{proof}
We apply \Cref{lem:momose_lemma_1} to $\gamma_\fq$ as defined in \Cref{def:ABC} and obtain
\begin{equation}\label{eq:mu_div}
\mu(\sigma_\fq)^{h_\fq} = \iota_{\fp_0}(\gamma_\fq^\eps) \Mod{\fp_0}.
\end{equation}
If $E$ has potentially multiplicative reduction at $\fq$, then $\mu(\sigma_\fq)$ is either $1$ or $\Nm(\fq)^{12} \Mod{p}$ (Proposition 1.4 in \cite{david2012caractere}, Proposition 3.3 in \cite{david2011borne}). Combining this with \Cref{eq:mu_div} yields (1).

If $E$ has potentially good reduction at $\fq$, then for $\mathcal{P}^\fq$ a prime of $L^\fq$ above $p$, the images of the roots of $P_\fq(X)$ in $\mathcal{O}_{L^\fq}/\mathcal{P}^\fq$ are in $\F_p^\times$, and there is a root $\beta_\fq$ of $P_\fq(X)$ such that $\mu(\sigma_\fq) = \beta_\fq^{12} \Mod{\mathcal{P}^\fq}$ (Proposition 1.8 in \cite{david2012caractere}, Proposition 3.6 in \cite{david2011borne}). Combining this with \Cref{eq:mu_div} yields (2).
\end{proof}

Loosely, \Cref{prop:ABC_div} says that the integers $A$, $B$, $C_o$ and $C_s$ may be considered as ``multiplicative nets'' to capture isogeny primes $p$ of signature $\eps$ for which the elliptic curve $E$ has a certain type of reduction at the prime $\fq$. We have defined $C$ because we will sometimes not be overly concerned with the distinction between potentially good ordinary and potentially good supersingular reduction types.

\begin{remark}\label{rem:order_in_class_group}
Momose's original definition of these integers used the class number $h_K$ instead of the order $h_\fq$ of $\fq$ in the class group $\Cl_K$, and this is how we initially implemented these integers in our program. Using $h_\fq$ instead of $h_K$ yields smaller integers and therefore faster runtime in the ensuing algorithm; this is now the implementation in the current version, to be described in \Cref{sec:parts_together}. We are grateful to Maarten Derickx and Aurel Page for independently suggesting this improvement to us.
\end{remark}

One is therefore led to consider the case that one or more of these integers is zero, since clearly this would not yield any meaningful multiplicative bound on $p$. The consideration of such cases yields the following key result.

\begin{proposition}[David, Momose]\label{prop:david_prop_215}
Let $q$ be a rational prime which splits completely in $K$, and let $\fq$ be a prime of $K$ over $q$. Let $h_\fq$ be the order of $\fq$ in $\Cl_K$, and let $\gamma_\fq$ be a generator of the principal ideal $\fq^{h_\fq}$. If the condition shown in the left-most column of \Cref{tab:david_types} is satisfied, then the corresponding assertions in the rest of the table hold.
\end{proposition}

{\small
\begin{table}[htp]
\begin{center}
\begin{tabular}{|c|c|c|c|c|c|c|}
\hline
Condition & $\gamma_\fq^\eps$ & $\Q(\beta)$ & $\gamma_\fq^\eps \in \Q$? & $(a_\tau)_{\tau \in G}$ & \thead{Everywhere \\ unramified\\ character} & \thead{Signature \\ Type}\\
\hline
$A(\eps, \fq) = 0$ & $1$ & & \multirow{4}{*}{Yes} & All $0$ & $\mu$ & \multirow{2}{*}{Type 1}\\\cline{1-2} \cline{5-6}
$B(\eps, \fq) = 0$ & $q^{12h_\fq}$ & & & All $12$ & $\mu/\chi_p^{12}$ & \\
\cline{1-3} \cline{5-7}
\multirow{2}{*}{$C_s(\eps, \fq) = 0$} & \multirow{2}{*}{$q^{6h_\fq}$} & \multirow{2}{*}{$\Q(\sqrt{-q})$} & & \multirow{2}{*}{All $6$} & \multirow{2}{*}{$\mu/\chi_p^{6}$} & \multirow{2}{*}{Type 2}\\
& & & & & &\\\hline
\multirow{6}{*}{$C_o(\eps, \fq) = 0$} & \multirow{6}{*}{\makecell{$\beta^{12h_\fq}$ for $\beta$ \\ an ordinary \\ Frobenius root \\ over $\F_\fq$}} & \multirow{6}{*}{\makecell{$\Q(\beta) \subseteq k$, \\ $\Nm_{K/\Q(\beta)}(\fq) = $ \\ $(\beta)$ or $(\overline{\beta})$}} & \multirow{6}{*}{No} & \multirow{3}{*}{\makecell{$a_\tau = 12$ for \\ $\tau \in \Gal(K/{\Q(\beta))}$; \\ $0$ otherwise}} & & \multirow{6}{*}{Type 3}\\
 & & & & & &\\
& & & & & & \\\cline{5-5}
& & & & \multirow{3}{*}{\makecell{$a_\tau = 0$ for \\ $\tau \in \Gal(K/{\Q(\beta))}$; \\ $12$ otherwise}} & & \\
& & & & & &\\
& & & & & &\\
\hline
\end{tabular}
\vspace{0.3cm}
\caption{\label{tab:david_types}Summary of what happens if one of the integers $A$, $B$, $C_s$ or $C_o$ is zero. The entries in the `Everywhere unramified character' column are to be interpreted as follows: if the condition in the `Condition' column holds, then for every elliptic curve over $K$ admitting a $K$-rational $p$-isogeny of isogeny character $\lambda$ with signature $\eps$ with $p \neq q$, the character listed in the `Everywhere unramified character' column is everywhere unramified. The column `Signature Type' is defining types of signatures according to the $(a_\tau)_{\tau \in G}$-column.}
\end{center}
\end{table}
}

\begin{proof}
This result is essentially Proposition 2.15 in \cite{david2012caractere}, which itself was considered as an explicit version of Lemma 2 of \cite{momose1995isogenies}. It is also given as Proposition 2.4.2 in \cite{david2008caractere}. We provide some of the details here.

Since $q$ splits completely in the Galois extensions $K/\Q$, the ideals $\tau(\fq)$ for $\tau \in G$ give the distinct ideals of $K$ lying above $q$, and their product is the ideal $q\O_K$.

Consider first that $A(\eps, \fq) = 0$, i.e., that $\gamma_\fq^\eps = 1$. By considering the ideal of $\O_K$ generated by $\gamma_\fq^\eps$, we obtain
\[ \O_K = \left(\prod_{\tau \in G}\tau(\fq)^{a_\tau}\right)^{h_\fq} \]
which implies that $a_\tau = 0$ for all $i$. Since $\mu|_{I_\fp} = \chi_p^{a_\fp}$ and $\mu$ is unramified away from $p$, we obtain that $\mu$ is unramified everywhere.

In the second case that $\gamma_\fq^\eps = q^{12h_\fq}$, one similarly obtains that $a_\tau = 12$ for all $\tau$, whence $\mu/\chi_p^{12}$ is everywhere unramified.

We consider the remaining cases together; that is, that $\gamma_\fq^\eps = \beta^{12h_\fq}$, where $\beta$ is a root of the characterisitic polynomial of Frobenius of a (supersingular or ordinary) elliptic curve over $\F_\fq$. We write $L = \Q(\beta)$, which is either $\Q$ or an imaginary quadratic field. By assumption, the element $\gamma_\fq^\eps$ - \emph{a priori} in $K$ - is also in $L$; so either $\gamma_\fq^\eps$ is rational, or it generates $L$ and therefore $L$ is contained in $K$.
 
In the first of these two subcases, $\beta^{12h_\fq}$ is rational; therefore it is equal to its complex conjugate $\bar{\beta}^{12h_\fq}$; in particular there is a $12h_\fq$\textsuperscript{th} root of unity $\zeta \in L$ such that $\beta = \zeta \bar{\beta}$. However, since $L$ is imaginary quadratic, it only admits $n$\textsuperscript{th} roots of unity for $n = 2, 4$ or $6$; thus $\beta^{12}$ is rational. Moreover, since $\beta$ is an algebraic integer, $\beta^{12}$ is an integer. Since the absolute value of $\beta$ is $\sqrt{\Nm({\fq})} = \sqrt{q}$, we get that $\beta^{12} = \pm q^6$, and therefore
\[ \left(\prod_{\tau \in G}\tau(\fq)^{a_\tau}\right)^{h_\fq}= \gamma_\fq^\eps\O_K = \beta^{12h_\fq}\O_K = q^{6h_\fq}\O_K = \left(\prod_{\tau \in G}\tau(\fq)\right)^{6h_\fq}.\]
This implies that $a_\tau = 6$ for all $\tau \in G$, which as before implies that $\mu/\chi_p^{6}$ is everywhere unramified.

Staying with this subcase, we show next that $\Q(\beta) = \Q(\sqrt{-q})$, and $\beta^{12h_\fq} = q^{6h_\fq}$.

The relation
\[ q = \Nm(\fq) = \bar{\beta}\beta = \zeta\beta^2 \]
implies that
\[ q\O_L = (\beta^2)\O_L = (\beta\O_L)^2; \]
i.e. that $q$ ramifies in $L$ (which in particular forces $L$ to be different to $\Q$.)

Write the characteristic polynomial of Frobenius of which $\beta$ is a root as $X^2 - T_\fq X + \Nm(\fq)$, and we consider here the case that $q \neq 2$. Since $q$ ramifies in $L$, $q$ divides the discriminant $T_\fq^2 - 4\Nm(\fq)$ of the polynomial. Since $q = \Nm(\fq)$, we obtain that $q$ divides $T_\fq$. Thus, $T_\fq^2$ is a non-negative integer which is simultaneously a square, a multiple of $q$, and (by the Hasse bound) is at most $4q$. This can only happen if $T_\fq = 0$, or $q = 3$, in which case $T_\fq^2 = 9$ is also a possibility.

In a similar way one may treat the case of $q = 2$, and reach the conclusion that $T_\fq$ must be zero. This is spelled out in David's proof, to which we refer the interested reader.

In the cases that $T_\fq = 0$, we obtain that $\beta$ is a root of the polynomial $X^2 + q$, i.e., that $\beta^2 = -q$, whence $\Q(\beta) = \Q(\sqrt{-q})$, and $\beta^{12h_\fq} = q^{6h\fq}$.

In the case that $q = 3$ and $T_\fq^2 = 9$, one obtains
\[ \beta = \frac{\pm3 \pm \sqrt{-3}}{2} \]
whence $\Q(\beta) = \Q(\sqrt{-q})$, and one may directly verify that $\beta^{12} = 3^6$.

In the second subcase, $L$ is an imaginary quadratic field contained in $K$. We here only sketch the approach, again leaving the details and work-through of the ideal equations to David's paper \cite{david2012caractere}. Let $H := \Gal(K/L)$. Since $\gamma_\fq \in L$, it is invariant under the action of all $\rho \in H$; this implies that for all $\tau \in G$, the integer $a_\tau$ is constant on coset (left or right) classes in $G/H$. Since there are only two classes - the identity $H$, and $H\gamma$, for $\gamma$ an element of $G$ which induces complex conjugation on $L$ - one obtains that half of the $a_\tau$ are one value $a_{id}$, while the other half are another value $a_\gamma$. Coupling this with the equation of $L$-ideals $(\beta^{12h_\fq})\O_L = \gamma_\fq\O_L$, one deduces that $\Nm_{K/L}(\fq)$ is principal, generated by either $\beta$ or $\bar{\beta}$. The consideration of each of these cases yields that either $a_{id} = 0$ and $a_\gamma = 12$, or vice-versa. This finishes the proof.
\end{proof}

The `Signature Type' column refers to the ``Types'' which Momose identified in Lemma 2 of \cite{momose1995isogenies} for the signature $\eps$ of an isogeny character $\lambda$, and are defined as follows.

\begin{definition}\label{def:}
We say that a signature $\eps$ of an isogeny character $\lambda$ is of \textbf{Type $1$} if the integers in $\eps$ are either all $0$ or all $12$.

We say it is of \textbf{Type $2$} if the integers in $\eps$ are all $6$.

We say it is \textbf{Type $3$ with field $L$} if $K$ contains an imaginary quadratic field $L$, and either $a_\tau = 12$ for $\tau \in \Gal(K/L)$ and $0$ otherwise, or	$a_\tau = 0$ for $\tau \in \Gal(K/L)$ and $12$ otherwise.
\end{definition}

We wish to emphasise that this definition \textbf{applies only to the signature $\eps$ of an isogeny character, and not \emph{a priori} to the isogeny character $\lambda$ itself}. While Momose did use the terms ``Type 1'', ``Type 2'' and ``Type 3'' to refer to the isogeny character $\lambda$ in the statement of his main theorem, these terms originally came from defining them for the signature, and as pointed out by the referee, it is a source of confusion to use the same terms for both the signature $\eps$ as well as the isogeny character $\lambda$. This is particularly so for Types 2 and 3, which are genuinely different in both cases (more on this later). For this reason we will use the terms \textbf{Momose type $N$} (for $N = 1,2,3$) to refer specifically to the three Types of $\lambda$ which Momose identified in the statement of his main theorem, which we recall as follows.

\begin{definition}
We say that an isogeny character $\lambda$ is of \textbf{Momose Type $1$} if either $\lambda^{12}$ or $(\lambda/\chi_p)^{12}$ is everywhere unramified.

We say it is of \textbf{Momose Type $2$} if $\lambda^{12} = \chi_p^6$.

We say it is of \textbf{Momose Type $3$} if $K$ contains an imaginary quadratic field $L$ as well as its Hilbert class field, $p$ splits in $L$ as $\fp_L\cdot\overline{\fp_L}$, and for any prime $\fq$ of $K$ coprime to $\fp_L$,
\[ \lambda^{12}(\sigma_\fq) = \alpha^{12} \Mod{\fp_L} \]
for any $\alpha \in L^\times$ a generator of $\Nm_{K/L}(\fq)$\footnote{Note that this is necessarily a principal ideal which follows from an application of the Principal ideal theorem.}.
\end{definition}

Note that it is clear from \Cref{tab:david_types} that $\eps$ being of signature Type 1 is equivalent to $\lambda$ being of Momose Type 1; however, as will be discussed in the sequel, Momose Types 2 and 3 are respectively stronger than signature types $2$ and $3$. Note also that David entirely avoids the use of these `Type' designations. In the sequel we will use the term `Type' only to refer to the signature, and where necessary, will use the term `Momose Type' for the isogeny character. Observe also that the (signature!) Type is independent of the choice of prime ideal $\fp_0$ above $p$, since this merely permutes the integers in the signature.

We make one final definition regarding these `Types'.

\begin{definition}
Let $K$ be a number field Galois over $\Q$. For $N = 1,2,3$, we say that a prime $p$ is a \textbf{Type $N$ prime for $K$} if there exists an elliptic curve over $K$ admitting a $K$-rational $p$-isogeny with isogeny signature of Type $N$.
\end{definition}

With this nomenclature, we may frame a direct corollary of the above \Cref{prop:david_prop_215}, which deals with isogeny signatures not of one of these three identified Types.

\begin{corollary}\label{cor:main_algorithmic_result}
Let $K/\Q$ be a Galois extension of group $G$, and $\fq | q$ a completely split prime ideal of $K$. Let $\eps$ be an element in $\left\{0,4,6,8,12\right\}^G$ which is not of Type $1$, $2$ or $3$. Then the integers $A(\eps, \fq)$, $B(\eps, \fq)$ and $C(\eps, \fq)$ are all non-zero. In particular, if there exists an elliptic curve over $K$ admitting a $K$-rational $p$-isogeny with isogeny character of signature $\eps$, then $p$ divides the non-zero integer $ABC(\eps, \fq)$ defined in \Cref{def:ABC}.
\end{corollary} 

\begin{proof}
The first assertion is a direct corollary of \Cref{prop:david_prop_215}. That the isogeny prime $p$ from the statement divides one of $A(\eps, \fq)$, $B(\eps, \fq)$ or $C(\eps, \fq)$ was seen previously under the assumption that $q \neq p$, and it for this reason that $\Nm(\fq)$ is also included in the $\lcm$.
\end{proof}

We are therefore reduced to studying isogenies with a signature of Type $1$, $2$ or $3$. Since one of the goals of this section is to explain Momose's Isogeny Theorem, which is stated in terms Momose Types, we now - to the extent required for the sequel - consider the Type $2$ and $3$ cases further.

Let therefore $\eps = 6\Nm(K/\Q)$. The following key result relates the notions of signature Type 2 and Momose Type 2.

\begin{proposition}\label{prop:type_2_to_momose_type_2}
Let $E/K$ be an elliptic curve admitting a $K$-rational $p$-isogeny of isogeny character $\lambda$ of signature $\eps$ of Type $2$. If there exists a set $\Gen$ of primes generating $\Cl_K$, such that for all $\fq \in \Gen$:
\begin{enumerate}
        \item $\fq$ is coprime to $p$;
        \item $\fq$ is a completely split prime ideal;
        \item $E$ has potentially good supersingular reduction at $\fq$,
\end{enumerate}
then $\lambda$ is of Momose Type 2.
\end{proposition}

\begin{proof}
The proof is modelled on the proof of Proposition 4.5 of \cite{david2012caractere}. Since $\eps$ is of Type 2 we have that $\mu\chi_p^{-6}$ is an everywhere unramified character, and hence defines an abelian extension of $K$ contained in the Hilbert class field of $K$, and thus is determined by its values at Frobenius automorphisms $\sigma_\fq$ for $\fq$ running through a set of generators of $\Cl_K$. Let $\fq \in \Gen$. From condition (2) we have that $\Nm(\fq)$ is a rational prime, call it $q$. From condition (3) we obtain $\lambda(\sigma_\fq) \equiv \beta_\fq \Mod{\mathcal{P}^\fq}$ for $\beta_\fq$ a supersingular Frobenius root over $\fq$. By the same reasoning as in the proof of \Cref{prop:david_prop_215} we have that $\beta_\fq^{12} = q^6 = \Nm(\fq)^6$, and hence
\begin{align*}
	\mu(\sigma_\fq) &= \Nm(\fq)^6 \Mod{p}\\
	 &= \chi_p(\sigma_\fq)^6.
	\end{align*}
Since this is true for all $\fq \in \Gen$, we obtain $\mu = \chi_p^6$ (i.e., not just equal up to an everywhere unramified character).
\end{proof}

\begin{corollary}\label{cor:type_2_not_momose}
Let $K$ be a number field Galois over $\Q$ with $G = \Gal(K/\Q)$, and let $\eps_6 := 6\sum_{\tau \in G}\tau$ be the Type 2 signature. Let $\Gen$ be a set of completely split prime ideal generators of $\Cl_K$. Then the integer
\[ ABC_o(\Gen) := \underset{\fq \in \Gen}\lcm\left(A(\eps_6, \fq), B(\eps_6, \fq), C_o(\eps_6, \fq), \Nm(\fq)\right)\]
is non-zero. Furthermore, if $E/K$ is an elliptic curve admitting a $K$-rational $p$-isogeny of isogeny character $\lambda$ of signature $\eps_6$, but $\lambda$ is not of Momose Type 2, then $p$ divides $ABC_o(\Gen)$.
\end{corollary}

\begin{proof}
That $ABC_o(\Gen)$ is non-zero follows from \Cref{prop:david_prop_215}; namely, if one of $A(\eps_6, \fq)$, $B(\eps_6, \fq)$ or $C_o(\eps_6, \fq)$ were zero, then one would conclude that $\eps_6$ were of Type 1 or 3 (which it is not). For the last claim, since $\lambda$ is assumed not of Momose Type 2, by \Cref{prop:type_2_to_momose_type_2}, there must exist a prime $\fq$ in $\Gen$ such that $E$ does \emph{not} have potentially good supersingular reduction at $\fq$; therefore $p$ must divide the integer $ABC_o(\Gen)$.
\end{proof}

\begin{remark}
Since every ideal class contains a completely split prime ideal if $K/\Q$ is Galois, the set $\Gen$ in the statement of the above corollary exists. (See e.g. Proposition 4.2 of \cite{david2012caractere}.)
\end{remark}

One then deals separately with isogenies of Momose Type 2; this will be done in \Cref{sec:type_2_primes}. Note that when $h_k = 1$ then we can take the empty set of generators in \Cref{cor:type_2_not_momose}, and hence we obtain that an isogeny of signature Type 2 is automatically of Momose Type 2.

\begin{remark}
David takes (D\'{e}finition 4.1 of \cite{david2012caractere}, Proposition 3.1.2 in \cite{david2008caractere}) as a generating set for $\Cl_K$ the set $\mathcal{J}_K$ of ideals of $K$ of norm at most $2(\Delta_K)^{Ah_K}$, for $A$ the absolute and effectively computable constant appearing in the (unconditional) Effective Chebotarev Density theorem (Theorem 1.1 in \cite{lagarias1979bound}). We now know that we can take $A = 12577$ by Theorem 1.1 of \cite{ahn2019explicit}, and we thank the referee for pointing this out to us. By Proposition 4.2 of \cite{david2012caractere}, the set $\mathcal{J}_K$ is indeed a generating set. In the proof of Momose's Theorem 1 he further assumes that the ideals in the set of generators for $\Cl_K$ are coprime to $6h_K$; this is so that he may conclude that $\beta_\fq = \pm \sqrt{-q}$; however as shown by David in the proof of Proposition 2.15 of \cite{david2012caractere} - specifically in the proof that $\mu(\sigma_\fq) = q^6 \Mod{p}$ - this is an unnecessary restriction.
\end{remark}

Suppose next that $\eps$ is of Type 3; in particular $K$ contains an imaginary quadratic field $L$. Both Momose and David show, for $p$ sufficiently large with respect to $K$, that an isogeny character of signature Type $3$ must be of Momose Type $3$. Analogously to the Type $2$ case, one may be more precise here and identify a non-zero integer which serves as a multiplicative bound on these ``signature Type 3 but not Momose Type 3'' isogeny primes. The details of this are rather slippery and will appear in forthcoming work with Derickx \cite{banwait2022explicit}. For now we observe from \Cref{tab:david_types} that the integers $A(\eps, \fq)$, $B(\eps, \fq)$ and $C_s(\eps, \fq)$ are non-zero for completely split prime ideals $\fq$, and we can ensure that $C_o(\eps, \fq)$ is non-zero provided we choose $\fq$ to furthermore be such that $\Nm_{K/L}(\fq)$ is \emph{not} a principal ideal of $L$. A positive density of such primes necessarily exist provided we assume that $K$ does not contain the Hilbert class field of $L$, which is sufficient for our purposes. (In effect, we are precluding the arising of Momose Type 3 isogenies.)

We may summarise the above discussion as the following sharpened version of Momose's Theorem A with the assumptions that suffice for our purposes in the sequel.

\begin{corollary}\label{cor:momose_weak_sharp}
Let $K/\Q$ be a Galois number field which does not contain the Hilbert class field of an imaginary quadratic field. Write $G = \Gal(K/\Q)$. Let $\fq$ be a completely split prime ideal of $K$ which furthermore satifies the property that, if $K$ contains an imaginary quadratic field $L$, then $\Nm_{K/L}(\fq)$ is not a principal ideal of $L$. Let $\Gen$ be a set of completely split prime ideal generators of $\Cl_K$. Let $\eps$ be an element of $\left\{0,4,6,8,12\right\}^G$. Then the integer $ABC_o(\Gen)$ is non-zero, and if $\eps$ is not of Type $1$ or Type $2$, then the integer $ABC(\eps, \fq)$ is non-zero.

Moreover, if $E/K$ is an elliptic curve admitting a $K$-rational $p$-isogeny with isogeny character $\lambda$ of signature $\eps$, with $\lambda$ not of Momose Type 1 or 2, then $p$ divides the non-zero integer $D$, where
\[
D = \begin{cases*}
ABC_o(\Gen) & \mbox{if $\eps$ is of Type 2}\\
ABC(\eps, \fq) & \mbox{otherwise.}
\end{cases*}
\]

\end{corollary}

\begin{proof}
From \Cref{tab:david_types}, if $\eps$ is not of Type 1 or 2, the only way $ABC(\eps, \fq)$ can be zero is if $\eps$ is of Type 3, in which case $C_o(\eps, \fq)$ would be zero; but this would imply that $\Nm_{K/L}(\fq)$ is a principal ideal of $L$, contradicting our assumption. That $ABC_o(\Gen)$ is non-zero is \Cref{cor:type_2_not_momose}. The final statements about the elliptic curve $E$ follow from \Cref{cor:main_algorithmic_result} and \Cref{cor:type_2_not_momose} (recalling that $\eps$ being of signature Type 1 is equivalent to $\lambda$ being of Momose Type 1).
\end{proof}

Note that the above result includes the assumption that $K/\Q$ is Galois, since this is used throughout David's work, although this is not assumed in the statement of Momose's Theorem 1, which we may now state as follows.

\begin{theorem}[Momose's Isogeny Theorem]\label{thm:momose_isogeny}
Let $K$ be a number field. Then there exists a constant $C_K$ such that, if there exists an elliptic curve $E/K$ admitting a $K$-rational $p$-isogeny for $p > C_K$, then the associated isogeny character is of Momose Type $1$, $2$, or $3$.
\end{theorem}

While Momose's Lemmas 1 and 2 do assume that $K$ is Galois over $\Q$, an argument is given in the proof of Theorem 1 to cover the non-Galois case, by passing to the Galois closure. As pointed out by the referee, the details which Momose gives to extend to the non-Galois case are wanting of the care and clarity that David brings in her treatment of Momose's work; such a treatment will appear in forthcoming work of the author with Derickx \cite{banwait2022explicit}.

As mentioned in the introduction, David made the constant $C_K$ explicit. For the details we refer the interested reader to Section 2.4 of \cite{david2012caractere} or Section 2.3 of \cite{david2008caractere}; here we content ourselves to give only the recipe for the upper bound, since we will use (a minor variant of) this for the implementation of the $\DLMV$ bound in \Cref{sec:parts_together}.

David makes the following notational definitions of entities associated to a number field $K$ Galois over $\Q$.
\begin{equation}\label{eqn:david_constant}
\begin{aligned}
d_K &= [K : \Q]\\
\Delta_K &= \mbox{discriminant of } K\\
h_K &= \mbox{class number of } K\\
R_K &= \mbox{regulator of } K\\
r_K &= \mbox{rank of unit group of }K\\
\delta_K &= \begin{cases*}
    \frac{\ln 2}{r_K + 1} & if $d_K = 1$ or $2$\\
    \frac{1}{1201}\left(\frac{\ln(\ln d_K)}{\ln d_K}\right)^3 & if $d_K \geq 3$
    \end{cases*}\\
C_1(K) &= \frac{r_K^{r_K+1}\delta_K^{-(r_K-1)}}{2}\\
C_2(K) &= \exp(24C_1(K)R_K)\\
C(K,n) &= (n^{12h_K}C_2(K) + n^{6h_K})^{4} \ \ \mbox{ for }n \in \ZZ.
\end{aligned}
\end{equation}
David then takes the upper bound to be $C(K, 2(\Delta_K)^{Ah_K})$, where $A$ is the absolute constant which appears in the Effective Chebotarev Density Theorem (Theorem 1.1 in \cite{lagarias1979bound}) (and as before, we now know that we may take $A = 12577$).

\section{Overview of Quadratic Isogeny Primes}\label{sec:overview}

In the rest of the paper $K$ will denote a quadratic field. We may denote an isogeny signature $\eps$ as a pair $(a,b)$ of integers, referring to $a\cdot\id + b\cdot\sigma$, with $\sigma$ being the non-trivial Galois automorphism. Since there are five choices for each of $a$ and $b$, one immediately obtains $25$ possible signatures, which we designate as follows.

\begin{definition}\label{lem:epsilon_determination}
We define the following four categories for the $25$ possible values of the group ring character $\eps$ for a quadratic field $K$ as follows:

\begin{enumerate}
	\item \textup{\textbf{Quadratic}} $\eps$: the $4$ pairs $(12a,12b)$ for $a,b \in \left\{0,1\right\}$;
	\item \textup{\textbf{Quartic}} $\eps$: 
	the $5$ pairs $(6a,6b)$ for $a,b \in \left\{0,1,2\right\}$, excluding the $4$ quadratic pairs;
	\item \textup{\textbf{Sextic}} $\eps$: the $12$ pairs $(4a,4b)$ for $a,b \in \left\{0,1,2,3\right\}$, excluding the $4$ quadratic pairs;
	\item \textup{\textbf{Mixed}} $\eps$: the $4$ pairs $(4,6)$, $(8,6)$, $(6,4)$, $(6,8)$.
\end{enumerate}
\end{definition}

From the discussion in \Cref{sec:prelims} the names here have been chosen to reflect the nature of the group of automorphisms $\Aut_{\overline{\F_p}}(\widetilde{E}_\fp)$, for $\fp$ a prime of $K$ above $p$ and $\widetilde{E}_\fp$ the reduction of $E$ at $\fp$ (which may be a singular curve in the bad reduction case).

We observe that we in fact do not need to consider all $25$ signatures. As explained in Remark 2 of \cite{momose1995isogenies}, if an elliptic curve $E/K$ has a $K$-rational $p$-isogeny with character of signature $\eps = (a,b)$, then the dual isogeny (which corresponds to the action of the Fricke involution $w_p$ on $X_0(p)$) will be a $K$-rational $p$-isogeny of signature $(12 - a, 12 - b)$. Since we are interested in bounding the $p$ that can possible arise, we need only consider signatures up to ``dualising'' under $\eps \mapsto 12 - \eps$. This brings the number of signatures to consider down to $13$.

Furthermore, since swapping the integers in the signature corresponds in the moduli interpretation to sending a $K$-point $x$ on $X_0(p)$ to its Galois conjugate, we need additionally only consider signatures up to this swapping. This further brings the number of signatures to consider down to a mere $10$. We are grateful to Michael Stoll for suggesting these improvements (to reduce the number of signatures) to us.

\begin{remark}\label{rem:mazur_papers}
A previous version of this paper wrongly assumed that there can be no elliptic curves with isogeny of mixed signature, however this is not the case. We are grateful to Maarten Derickx for pointing out this gap in our previous argument. Explicit examples of such isogenies will be appearing in forthcoming work of the author with Derickx \cite{banwait2022explicit}.
\end{remark}

We next restate \Cref{cor:momose_weak_sharp} in the quadratic case, which simplifies to the following.

\begin{corollary}\label{cor:quadratic_algo}
Let $K/\Q$ be a quadratic field which is not imaginary quadratic of class number one. Let $\eps$ be a possible isogeny signature which is not of Type $1$ or $2$. Let $\fq$ be a split prime ideal of $K$ which, if $K$ is imaginary quadratic, is furthermore taken to be not principal. Then the integer $ABC(\eps, \fq)$ is non-zero. Furthermore if $E/K$ is an elliptic curve admitting a $K$-rational $p$-isogeny with isogeny signature $\eps$, then $p$ divides $ABC(\eps, \fq)$.
\end{corollary}

This result suggests that one deal separately with signature Types 1 and 2. These will be dealt with in \Cref{sec:type_1_primes} and \Cref{sec:type_2_primes} respectively. In this remainder of this section we consider signatures $\eps$ not of Type 1 or 2. This will yield a finite and explicitly computable superset of the isogeny primes whose associated isogeny character has a signature which is not of Type 1 or 2, and we therefore denote this superset by $\NotTypeOneTwoPrimes(K)$.

Let therefore $\eps$ be a signature not equal to $(0,0)$ (the only Type 1 case remaining after identifying signatures) or $(6,6)$ (the Type 2 case). We refer to prime ideals $\fq$ as in \Cref{cor:quadratic_algo} as \textbf{auxiliary primes}, viz. split prime ideals in $K$ lying above rational primes $q$, which in the imaginary quadratic case are assumed to be \emph{not principal}. We take a finite set $\Aux$ of such auxiliary prime ideals, and we compute a superset of the non-Type 1 or 2 primes for $K$ as 
\begin{equation}\label{eqn:pre_type_one_two_primes}
\NotTypeOneTwoPrimes(K) := \bigcup_{\eps} \ \Supp \left( \gcd_{\fq \in \Aux}(ABC(\eps, \fq))\right),
\end{equation}
where the union is taken over the $8$ pairs
\[ (0,12), (0,4), (0,8), (4,4), (4,8), (4,12), (4,6), (0,6),\]
and $\Supp(\cdot)$ refers to taking the set of primes divisors of the argument.

We summarise this re-interpretation of Momose's Theorem 1 into the following.

\begin{proposition}
Let $K$ be a quadratic field which is not imaginary quadratic of class number one, and $p$ a prime not in $\NotTypeOneTwoPrimes(K)$. Then, if there exists an elliptic curve $E/K$ admitting a $K$-rational $p$-isogeny, the associated isogeny character has a signature of Type 1 or Type 2.
\end{proposition}

The implementation of $\NotTypeOneTwoPrimes(K)$ in the function \path{get_not_type_one_two_primes} includes the necessary conditions discussed after the statement of \Cref{lem:momose_lemma_1}. Namely, if a prime $p$ appears in $\NotTypeOneTwoPrimes(K)$ from a signature $\eps$, then the following conditions must be satisfied in order for $p$ to appear in the output of $\NotTypeOneTwoPrimes(K)$:

\begin{itemize}
	\item
	if $\eps$ is not constant, then $p$ is required to split in $K$;
	\item
	if $\eps$ is sextic, then $p$ is required to be $2 \Mod{3}$;
	\item
	if $\eps$ is quartic, then $p$ is required to be $3 \Mod{4}$;
	\item
	if $\eps$ is mixed, then $p$ is required to be $11 \Mod{12}$.
\end{itemize}

This takes place in the function \path{filter_ABC_primes}.

\section{Computing Type 1 primes}\label{sec:type_1_primes}
Momose proved (Theorem $3$ in \cite{momose1995isogenies}) that a number field of degree at most $12$ admits only finitely many Type 1 primes. At that time, the restriction on degree arose from the $n$ for which the natural\footnote{in the context of Kamienny's extension of Mazur's formal immersion method} map from the smooth locus of the $n$\textsuperscript{th} symmetric power of $X_0(p)$ into the N\'{e}ron model of the Eisenstein quotient of $J_0(p)$ was known to be a formal immersion. This restriction on degree was removed very soon after Momose's paper appeared, as a result of Merel's resolution of the uniform boundedness conjecture for torsion on elliptic curves \cite{merel1996bornes}. Furthermore, not long after Merel's work appeared, Oesterl\'{e} provided a very explicit bound on the torsion primes that arise over a number field of given degree. Oesterl\'{e} never published his result; this gap in the literature was filled by Derickx, Kamienny, Stein, and Stoll as Appendix~A in \cite{derickx2019torsion}. 

Both David as well as Larson and Vaintrob use this bound of Oesterl\'{e} to give an explicit bound on the Type 1 primes as follows. If an elliptic curve $E/K$ over a number field of degree $d_K$ admits a $K$-rational $p$-isogeny with isogeny character $\lambda$ of Momose Type 1, then, by possibly replacing $E$ with its isogenous curve, we may suppose that $\mu := \lambda^{12}$ is unramified everywhere. The kernel field $K^\mu$ of $\mu$ is thus contained inside the Hilbert class field of $K$ and hence has degree at most $d_Kh_K$. Therefore, the kernel field $K^\lambda$ of $\lambda$ is a number field of degree at most $12d_Kh_K$. By construction, the $\Gal(\overline{K}/K^\lambda)$-action on the kernel of the isogeny is trivial, meaning that $E$ admits a $K^\lambda$-rational $p$-torsion point, and hence one may apply the Oesterl\'{e} bound to conclude that $p \leq (1 + 3^{6d_Kh_K})^2$. We will use this in \Cref{sec:parts_together} when we discuss the implementation of the $\DLMV$ bound.

In this section, we describe an algorithm to determine a superset $\TypeOnePrimes(K)$ of the set of Type 1 primes in the case of quadratic $K$. This is based on the proof of Momose's Theorem 3, and recall that we need only deal with the case of $\eps = (0,0)$ (since $(12,12)$ corresponds to the dual isogeny).

For $\fq | q$ a split prime of $K$ whose order in $\Cl_K$ is $h_\fq$, define the integer
\[ D(\fq) := \lcm\left(q, q^{12h_\fq} - 1, C((0,0), \fq)\right).\]

The main result of this section is as follows.

\begin{proposition}\label{prop:type_1_primes}
Let $\Aux$ denote a finite set of prime ideals of $K$ which are split and which lie over rational primes greater than $5$. Define the following set.
\[ \TypeOnePrimes(K) = \PrimesUpTo(61) \cup \left\{71 \right\} \cup \Supp\left( \gcd_{\fq \in \Aux}(D(\fq))\right).\]
Then $\TypeOnePrimes(K)$ is a superset for the set of Type 1 primes of $K$.
\end{proposition}

\begin{remark}
One may do better than having to include all primes $\leq 61$ and $71$ to $\TypeOnePrimes(K)$ by applying the sharper formal immersion bounds to be found in \cite{derickx2019torsion}. See forthcoming work with Derickx \cite{banwait2022explicit} where this is applied to obtain an algorithm to compute a superset for $\IsogPrimeDeg(K)$ for a general number field (not containing the Hilbert class field of an imaginary quadratic field).
\end{remark}

\begin{proof}
As in \Cref{sec:prelims} we take $q \neq p$ a rational prime which splits in $K$, and $\fq | q$ an auxiliary prime. We apply \Cref{prop:david_prop_215} to obtain that $C((0,0), \fq) \neq 0$, which by \Cref{prop:ABC_div} gives a multiplicative bound on the primes $p$ for which $E$ has potentially good reduction at $\fq$.

We are thus reduced to considering $E$ having potentially multiplicative reduction at $\fq$. That is, writing $x$ for the non-cuspidal $K$-point on $X_0(p)$ corresponding to $E$, we have that $x$ reduces modulo $\fq$ to one of the two cusps $0$ or $\infty$.

Suppose first that $x$ reduces modulo $\fq$ to the zero cusp. Following Kamienny's proof of Lemma 3.3 of \cite{kamienny1992torsion}, this means that the kernel $W$ of the isogeny specialises to the identity component $(E_{/\F_{\fq}})^0$ of the reduction of $E$ at $\fq$. Since this coincides with the group scheme $\mu_p$ of $p$\textsuperscript{th} roots of unity over an (at most) quadratic extension of $\F_{\fq}$, we obtain an equality of groups between $W(\overline{K_\fq})$ and $\mu_p(\overline{K_\fq})$, where $K_\fq$ denotes the completion of $K$ at $\fq$. This precise case is considered by David in the proof of Proposition 3.3 of \cite{david2011borne}, and corresponds to the isogeny character satisfying $\lambda^2(\sigma_\fq) \equiv \Nm(\fq)^2 \Mod{p}$. However, since $\eps = (0,0)$ we obtain that $\lambda^{12h_\fq}$ acts trivially on $\sigma_\fq$, yielding the divisibility 
\[ p | q^{12h_\fq} - 1.\]

Suppose next that $x$ reduces modulo $\fq$ to the infinite cusp. At this point we consider the Galois conjugate point $x^\tau$, for $\tau$ the non-trivial Galois automorphism of $\Gal(K/\Q)$ and observe that, if $x^\tau$ specialises to the zero cusp modulo $\fq$, then we may apply the previous argument to the conjugate curve $E^\tau$ (which also has isogeny signature $(0,0)$) and conclude as before that $p | q^{12h_\fq} - 1$.

We are thus reduced to considering the case that the pair $(x,x^\tau)$, which one may view as an $S$-section of the symmetric square $X_0(p)^{(2)}$ (which we regard as a smooth scheme over $S := \Spec(\ZZ[1/p])$), ``meets'' the section $(\infty, \infty)$ at $\fq$ (to use Kamienny's terminology in the discussion immediately preceding his Theorem 3.4 in \cite{kamienny1992torsion}; Mazur's language for this in the proof of Corollary 4.3 of \cite{mazur1978rational} is that these two sections ``cross'' at $\fq$). By applying Kamienny's extension of Mazur's formal immersion argument (again, to be found just before Theorem 3.4 in \cite{kamienny1992torsion}), one obtains that the map
\begin{alignat*}{2}
    f^{(2)}_{p} : \eqmathbox{X_0(p)^{(2)}_{/S}} &\longrightarrow \eqmathbox{J_0(p)_{/S}} & &\longrightarrow \eqmathbox{\tilde{J}_{/S}}\\
    \eqmathbox{D} &\longmapsto \eqmathbox{[D - 2(\infty)]} & &\longmapsto [D - 2(\infty)] \Mod{\gamma_{\mathfrak{J}}J_0(p)}
\end{alignat*}
is \emph{not} a formal immersion along $(\infty, \infty)$ in characteristic $q$. (Here $\tilde{J}$ is the Eisenstein quotient of $J_0(p)$ and $\gamma_{\mathfrak{J}}$ is a certain ideal of the Hecke algebra $\mathbf{T}$ associated to the Eisenstein ideal $\mathfrak{J}$. See Chapter 2, Section 10 of \cite{Mazur3} for more details.) Thus, if we choose $q$ to be strictly larger than $5$, Proposition 3.2 of \cite{kamienny1992torsion} implies that $p \leq 61$ or $p = 71$.
\end{proof}

\begin{remark}
Note that the multiplicative bound $p | q^{12h_\fq} - 1$ in the above proof is different to the bound Momose gives in his proof for this case of reduction to the zero cusp. Momose concludes that $p - 1 | 12h_K$, and this is arrived at through the claim that ``the restriction of $\lambda$ to the inertial subgroup $I_\fq$ of $\fq$ is $\pm \theta_p$'' ($\theta_p$ being Momose's notation for the mod-$p$ cyclotomic character). Momose cites the whole of \cite{delignerapoport} for this claim, which alas we were unable to find in those $174$ pages. If this claim were true, then, since we also know that $\lambda^2$ is unramified at $\fq$, we would obtain that $p - 1 | 2$, i.e., that $p = 2$ or $3$. However we prefer to remain agnostic about this claim and apply David's more careful treatment as above. We are grateful to Maarten Derickx for highlighting this issue in Momose's argument to us.
\end{remark}

\section{Computing Type 2 primes}\label{sec:type_2_primes}

Let $E/K$ be an elliptic curve admitting a $K$-rational $p$-isogeny whose isogeny character $\lambda$ has signature $\eps = (6,6)$. Note from \Cref{sec:prelims} that, if $h_K > 1$, this does not necessarily mean that $\lambda$ is of Momose Type 2. The first step in the computation of Type 2 primes is therefore to compute a finite set of primes which contains the $p$ for which $\lambda$ has signature $(6,6)$ but which is not itself of Momose Type 2.

From \Cref{cor:momose_weak_sharp} such a bound is given by $ABC_o(\Gen)$ for $\Gen$ a set of completely split prime ideal generators of $\Cl_K$. Observe however that one may take several such $\Gen$ sets. Let $\AuxGen$ be a finite set of such generating sets of $\Cl_K$; then as before this allows one to construct a multiplicative sieve for $p$:
\[ \TypeTwoNotMomosePrimes(K) := \Supp \left(\gcd_{\Gen \in \AuxGen}ABC_o(\Gen)\right).\]
This is implemented in \path{get_type_2_not_momose}.

Having computed this set of primes, we are reduced to dealing with isogeny characters of Momose Type 2, which we consider for the remainder of this section. 

We first recall Momose's Lemmas 3, 4, and 5, with some clarifications in both the statements as well as the proofs. We thank the referee for supplying the details here.

\begin{lemma}[Momose, Lemmas~3-5 in \emph{loc. cit.}]\label{lem:mom_lem_3-5}
Let $E/K$ be an elliptic curve with a $K$-rational $p$-isogeny of isogeny character $\lambda$ of Momose Type 2. Then we have the following.
\begin{enumerate}
	\item
	$\lambda$ is of the form
\[ \lambda = \psi \chi_p^{\frac{p+1}{4}}\]
for a character $\psi$ of order dividing $6$.
	\item
	Let $\fq$ be a prime of $K$ lying over a rational prime $q \neq p$. If $E$ has potentially multiplicative reduction at $\fq$, and $q$ splits in the imaginary quadratic field $\Q(\sqrt{-p})$, then
\[ \Nm(\fq) \equiv \psi(\sigma_\fq)^{\pm2} \Mod{p}.\]
	\item
	Let $\fq$ be a prime of $K$ of odd residue field degree lying over a rational prime $q \neq p$. If $E$ has potentially good reduction at $\fq$, and $\Nm(\fq) < p/4$, then $q$ is inert in $\Q(\sqrt{-p})$.
\end{enumerate}
\end{lemma}

\begin{proof}
For (1), see Proposition 4.5 Part (3) of \cite{david2012caractere}.

For (2), by Proposition 1.4 Part (2) of \cite{david2012caractere} we have that $\lambda^2(\sigma_\fq) \equiv 1$ or $\Nm(\fq) \Mod{p}$. By (1), we thus obtain
\begin{equation}\label{eqn:mom_lem_4}
\psi(\sigma_\fq)^2\Nm(\fq)^{\frac{p+1}{2}} \equiv 1 \mbox{  or  } \Nm(\fq)^2 \Mod{p}.
\end{equation}
We now use the assumption that $q$ splits in $\Q(\sqrt{-p})$, which, by using quadratic reciprocity and that $p \equiv 3 \Mod{4}$, is equivalent to $q^{\frac{p+1}{2}} \equiv q \Mod{p}$; this simplifies \Cref{eqn:mom_lem_4} above to
\[ \psi(\sigma_\fq)^2\Nm(\fq) \equiv 1 \mbox{  or  } \Nm(\fq)^2 \Mod{p};\]
considering each of these cases in turn yields the result.

For (3), by Proposition 1.8 of \cite{david2012caractere} (or Proposition 3.6 in \cite{david2011borne}) we have the existence of a root $\beta$ of the characteristic polynomial of Frobenius of an elliptic curve over the residue field of $k$ at $\fq$ such that $\lambda(\sigma_\fq) \equiv \beta \Mod{p}$. Applying (1) we obtain
\[ \left(\psi(\sigma_\fq)^2 + \psi(\sigma_\fq)^{-2}\right)\Nm(\fq)^{\frac{p+1}{2}} \equiv \beta^2 + \bar{\beta}^{2} \Mod{p}.\]
If $q$ splits in $\Q(\sqrt{-p})$, then as in the proof of (2) we obtain that $q^{\frac{p+1}{2}} \equiv q \Mod{p}$, and hence $\Nm(\fq)^{\frac{p+1}{2}} \equiv \Nm(\fq) \Mod{p}$. Observing in addition that $\psi(\sigma_\fq)^2$ is a cube root of unity in $\F_p^\times$, the above equation becomes
\[ (\beta + \bar{\beta})^2 \equiv \Nm(\fq) \mbox{  or  } 4\Nm(\fq) \Mod{p}.\]
From the Hasse bound we have that $|\beta + \bar{\beta}| \leq 2\sqrt{\Nm(\fq)}$, and then the assumption that $\Nm(\fq) < p/4$ implies that $\beta + \bar{\beta}$ is either $\pm\sqrt{\Nm(\fq)}$ or $\pm2\sqrt{\Nm(\fq)}$. Since $\fq$ is assumed to have odd residue field degree, neither of these are integers, yielding the desired contradiction.
\end{proof}

Momose then states without proof that these three lemmas imply that a Momose Type~2 prime must satisfy the following condition.

\begin{condC}
Any rational prime $q$ with $q < p/4$ does not split in $K(\sqrt{-p})$, unless $q^2 + q + 1 \equiv 0 \Mod{p}$.
\end{condC}

It is not clear where the name of ``Condition C'' comes from. However, it is very similar to Mazur's ``Claim'' in the proof of Theorem 7.1 (the main theorem) in \cite{mazur1978rational}:

\begin{claim}
If the above case occurs then for all odd primes $p < N/4$ we have $\genfrac(){0pt}{1}{p}{N} = -1$.
\end{claim}

Therefore, we suspect that Momose was thinking of Condition C as analogous to Mazur's claim, and that the `C' stands for `Claim'.

Our initial Sage implementation took Condition C as given in Momose's paper. However, upon running the subsequent implementation on the number field $\Q(\sqrt{5})$, we were dismayed to find that the prime $163$ was not present (we shall discuss more about our verification tests for the algorithm in \Cref{sec:parts_together}). Since for Mazur this prime appears from showing that his Claim implies that $\Q(\sqrt{-N})$ has class number $1$, it suggested to us that we ought therefore to verify that Condition C indeed follows from Lemmas 3,4 and 5 as claimed.

In attempting this, we found that the Lemmas implied a very similar but weaker necessary condition that a Momose Type 2 prime must satisfy, for which we can find no better name than ``Condition CC''.

\begin{condCC}\hypertarget{conditionCC}
Let $K$ be an quadratic field, and $E/K$ an elliptic curve admitting a $K$-rational $p$-isogeny with isogeny character of Momose Type 2. Let $q$ be a rational prime $< p/4$ such that $q^2 + q + 1 \not\equiv 0 \Mod{p}$. Then the following implication holds:
\begin{center}
if $q$ splits or ramifies in $K$, then $q$ does not split in $\Q(\sqrt{-p})$.
\end{center}
\end{condCC}

\begin{proof}
Let $q$ be a rational prime strictly less than $p/4$, and let $\fq|q$ be a prime lying over $q$. Assume that $q$ splits or ramifies in $K$. If $E$ has potentially good reduction at $\fq$, then by part (3) of \Cref{lem:mom_lem_3-5} we have that $q$ is inert in $\Q(\sqrt{-p})$. If $E$ has potentially multiplicative reduction at $\fq$, and $q$ splits in $\Q(\sqrt{-p})$, then we proceed as follows.
\begin{align*}
&\ N(\fq) \equiv \psi(\sigma_\fq)^{\pm2} \Mod{p} \mbox{ (by Part (2) of \Cref{lem:mom_lem_3-5}) }\\
\Rightarrow &\ N(\fq) \mbox{ is a 3\textsuperscript{rd} root of unity in } \F_p^\times \mbox{ (since $\psi$ has order dividing $6$ by Part (1) of \Cref{lem:mom_lem_3-5}) }\\
\Leftrightarrow &\ q \mbox{ is a 3\textsuperscript{rd} root of unity in } \F_p^\times \mbox{, or $q$ is inert in $K$ (since $N(\fq)$ is either $q$ or $q^2$) }\\
\Leftrightarrow &\ q^2 + q + 1 \equiv 0 \Mod{p} \mbox{ or $q$ is inert in $K$ (since $2 \leq q < p/4$) }.
\end{align*}
We conclude in all cases that if $q$ is not inert in $K$, and $q^2 + q + 1 \not\equiv 0 \Mod{p}$, then $q$ is inert in $\Q(\sqrt{-p})$.
\end{proof}

Implementing this condition instead of Momose's Condition C then did yield $163$ in the final output for $\Q(\sqrt{5})$, as must necessarily be the case (since $163 \in \IsogPrimeDeg(\Q)$). Note that Condition CC also resembles more closely the conditions that Goldfeld considers in his Appendix to Mazur's paper \cite{mazur1978rational}. In summary, we obtain an explicit and unconditional counterexample to Momose's claim that Lemmas 3-5 imply Condition C is satisfied, since the non-cuspidal $\Q$-rational point on $X_0(163)$ corresponds to a Momose Type 2 isogeny and hence satisfies Lemmas 3 to 5 of Momose's paper, but $163$ does \emph{not} satisfy Momose's Condition C.

\begin{remark}
The referee pointed out that Momose probably meant to write ``split \underline{completely}'' rather than merely ``split'' in the statement of his Condition C. With this interpretation, we see that Condition C is equivalent to Condition CC, since a prime splitting completely in two number fields implies that it splits completely in their compositum. Implementing this corrected Condition C then does yield $163$ in the final output for $\Q(\sqrt{5})$. Note furthermore that Momose similarly omits this (rather important) word in his statement of Goldfeld's Conjecture; the original in the appendix to Mazur's paper \cite{mazur1978rational} does state this correctly. We are grateful to the referee for their detailed explanation of this point.
\end{remark}

We would now like to determine the primes $p$ which satisfy Condition~CC; that is, the primes $p$ such that, for all primes $q < p/4$ and $q^2 + q + 1 \not\equiv 0 \Mod{p}$, if $q$ splits or ramifies in $K$, then $q$ does not split in $\Q(\sqrt{-p})$. We denote this set as $\TypeTwoMomosePrimes(K)$. As mentioned above, in Mazur's case for $K=\Q$, he was able to show - in the very last paragraph of \S7 of \cite{mazur1978rational}, using nothing more than undergraduate algebraic number theory - that his $p$'s are such that $\Q(\sqrt{-p})$ has class number 1, so one can conclude by applying the theorem of Baker-Heegner-Stark.

This explicit determination is more difficult for higher degree number fields $K$. Indeed, in \S8 of \cite{mazur1978rational}, Mazur extends his results to the setting of imaginary quadratic fields $K$ and isogeny primes $N$ which are inert in $K$; and the analogous finiteness here (of Mazur's set $\mathcal{N}_4(K)$) is the subject of Goldfeld's Appendix, and is not effective. Similarly, Momose's Proposition 1 (which itself is based on Goldfeld's theorem) is not effective.

It was Larson and Vaintrob who found a bound on Momose Type 2 primes assuming GRH, which depended on an effectively computable absolute constant which alas was not effectively computed.

By building upon the methods of Larson and Vaintrob, and utilising the best possible bounds in the Effective Chebotarev Density Theorem due to Bach and Sorenson (Theorem 5.1 in \cite{bach1996explicit}), we are able to provide a completely explicit upper bound on Momose Type 2 primes.

\begin{proposition}\label{prop:type_2_bound}
Assume GRH. Let $K$ be an quadratic field, and $E/K$ an elliptic curve possessing a $K$-rational $p$-isogeny with isogeny character of Momose Type 2. Then $p$ satisfies
\[ p \leq (16\log p + 16\log(12\Delta_K) + 26)^4.\]
In particular, there are only finitely many primes $p$ as above.
\end{proposition}

\begin{proof}
In the proof of Theorem~6.4 of \cite{larson_vaintrob_2014}, the authors prove that a Momose Type~2 prime $p$ satisfies
\begin{equation}\label{eqn:type_2_ineq}
p \leq \left(1 + \sqrt{\Nm^K_\Q(v)} \right)^4
\end{equation}
where $v$ is a prime ideal of $K$ such that $v$ is split in $K(\sqrt{-p})$, is of degree $1$, does not lie over $3$, and satisfies the inequality
\[ \Nm_{K/\Q}(v) \leq c_7 \cdot (\log \Delta_{K(\sqrt{\pm p})} + n_{K(\sqrt{\pm p})} \log 3 )^2 \]
for an effectively computable absolute constant $c_7$ (note that they use $\ell$ instead of $p$, and $n_K$ denotes the degree of $K$).

The existence of such a $v$ follows from their Corollary 6.3, which requires GRH. Stepping into the proof of this Corollary, we arrive at a point where they apply the Effective Chebotarev Density Theorem to $\Gal(E'/K)$ - for a certain Galois extension $E'$ which fits into a tower of successive quadratic extensions $E'/E/K/\Q$ - to bound the norm of $v$ as
\[ \Nm_{K/\Q}(v) \leq c_5(\log \Delta_{E'})^2 \]
for an effectively computable absolute constant $c_5$.

However, we may at this breakpoint instead use Theorem 5.1 of \cite{bach1996explicit} on the degree $8$ extension $E'$ to obtain
\[ \Nm_{K/\Q}(v) \leq (4\log \Delta_{E'} + 25)^2. \]
This then subsequently (stepping back out into the proof of Theorem 6.4) yields the bound
\[ \Nm_{K/\Q}(v) \leq (16\log \Delta_{K} + 16\log p + 16\log 6 + 16\log 2 + 25)^2. \]
Inserting this into \Cref{eqn:type_2_ineq} yields the result.
\end{proof}

\begin{corollary}

Assume GRH. Then there is an algorithm to compute the set $\TypeTwoMomosePrimes(K)$.
\end{corollary}

\begin{proof}
For each $p$ up to the bound, we check whether it satisfies Condition CC, which is a finite computation.
\end{proof}

The above Proposition is also used in the implementation of the $\DLMV$ bound. 

\begin{example}
For $K = \Q(\sqrt{5})$, this bound on Momose Type 2 primes is approximately $5.61 \times 10^{10}$. The algorithm therefore needs to check all primes up to this large bound. See the User's guide of \emph{Quadratic Isogeny Primes} \cite{isogeny_primes} for the implementation details.
\end{example}

We end this section with some observations on whether we can remove the dependence on GRH.

While it is unconditionally true that $\IsogPrimeDeg(K)$ is finite for quadratic fields not imaginary quadratic of class number one, the effective bound in \Cref{prop:type_2_bound} above requires GRH to work; that is, replacing the bound on the smallest (ordered by norm) prime ideal in the above proof with the best known unconditional bound (Theorem 1.1 in \cite{lagarias1979bound}) results in a $p^{8A}$ term (for $A$ an effectively computable absolute constant which we can take to be $12577$) in the right-hand side of \Cref{eqn:type_2_ineq}, which clearly does not yield a contradiction for large $p$. Therefore, in this quadratic setting, finiteness of Type 2 primes is unconditionally true ineffectively, but effectively true conditionally.

Note that Larson publicly asked this very question about removing GRH \cite{larson_MO} around about the time his joint work with Vaintrob appeared. The feeling among the participants there was that this ``is pretty hopeless with present day technology'' (to quote prominent Mathoverflow contributor ``GH from MO'').

\section{Putting the parts together}\label{sec:parts_together}

We summarise the results from \Cref{sec:prelims,sec:overview,sec:type_1_primes,sec:type_2_primes} into the following. For the reader's convenience we restate the definitions.

\begin{proposition}\label{prop:main_detailed}
Let $K$ be a quadratic number field with discriminant $\Delta_K$, class group $\Cl_K$, class number $h_K$, and absolute Galois group $G_K$. Let $E/K$ be an elliptic curve admitting a $K$-rational $p$-isogeny for $p \geq 17$ prime, and denote the mod-$p$ isogeny character by $\lambda : G_K \to \F_p^\times$. Assume that $p$ is unramified in $K$. Then:

\begin{enumerate}
	\item
	For a fixed prime ideal $\fp_0$ of $K$ dividing $p$ there exists a unique $\eps = (a,b) \in \left\{0,4,6,8,12\right\}^2$ such that for every principal ideal of $K$ coprime to $p$, viz.
	\[ (\alpha) = \prod_{\fq \nmid p}\fq^{\ord_\fq(\alpha)}, \]
	we have
	\[ \prod_{\fq \nmid p}\lambda^{12}(\sigma_\fq)^{\ord_\fq(\alpha)} \equiv \iota_{\fp_0}(\alpha^a(\tau(\alpha))^b) \Mod{\fp_0},\]
	where $\tau$ denotes the non-trivial Galois automorphism in $\Gal(K/\Q)$, and $\sigma_\fq$ denotes the Frobenius automorphism at $\fq$. We henceforth write $\alpha^\eps := \alpha^a(\tau(\alpha))^b$, and refer to $\eps$ as the signature of $\lambda$. We furthermore define, for $\fq$ a prime of $K$, the integers
	\begin{align*}
	A(\eps, \fq) &:= \Nm_{K/\Q}(\gamma_\fq^\eps - 1)\\
	B(\eps, \fq) &:= \Nm_{K/\Q}(\gamma_\fq^\eps - \Nm(\fq)^{12h_\fq})\\
	C_s(\eps, \fq) &:= \lcm(\left\{ \Nm_{K(\beta)/\Q}(\gamma_\fq^\eps - \beta^{12h_\fq}) \ | \ \beta \mbox{ is a supersingular Frobenius root over }\F_\fq\right\})\\
	C_o(\eps, \fq) &:= \lcm(\left\{ \Nm_{K(\beta)/\Q}(\gamma_\fq^\eps - \beta^{12h_\fq}) \ | \ \beta \mbox{ is an ordinary Frobenius root over }\F_\fq\right\})\\
	C(\eps, \fq) &:= \lcm\left(C_o(\eps, \fq), C_s(\eps, \fq)\right),
	\end{align*}
	where $h_\fq$ is the order of $\fq$ in $\Cl_K$, $\gamma_\fq$ is a generator of $\fq^{h_\fq}$, and by a (supersingular, respectively ordinary) Frobenius root over $\F_\fq$ we mean a root of the characteristic polynomial of Frobenius of a (supersingular, respectively ordinary) elliptic curve defined over $\F_\fq$.
	\item
	If $\eps \neq (0,0), (12,12)$ or $(6,6)$, then for all split primes $\fq$ of $K$ coprime to $p$,
	\begin{itemize}
		\renewcommand\labelitemi{--}
		\item
		if $E$ has potentially multiplicative reduction at $\fq$, then $p$ divides one of the non-zero integers $A(\eps, \fq)$ or $B(\eps, \fq)$;
		\item
		if $E$ has potentially good reduction at $\fq$, then $p$ divides $C(\eps, \fq)$. Moreover,
			\begin{itemize}
				\renewcommand\labelitemii{$\bullet$}
				\item
				If $K$ is real quadratic, then $C(\eps, \fq) \neq 0$.
				\item
				If $K$ is imaginary quadratic, then $C(\eps, \fq) \neq 0$ if $\fq$ is not principal.
			\end{itemize}
	\end{itemize}
	\item
	(Type 1 case) If $\eps = (0,0)$ or $(12,12)$, then $\lambda^{12}$ or $(\lambda/\chi_p^{-1})^{12}$ is unramified everywhere, and for all split primes $\fq$ of $K$ coprime to $p$,
	\begin{itemize}
		\renewcommand\labelitemi{--}
		\item
		if $E$ has potentially multiplicative reduction at $\fq$, then either $p$ divides $\Nm(\fq)^{12h_\fq} - 1$, or $\Nm(\fq) \geq 7$ and $p \leq 61$ or $p = 71$;
		\item
		if $E$ has potentially good reduction at $\fq$, then $p$ divides the non-zero integer $C((0,0), \fq)$.
	\end{itemize}
	\item
	(Type 2 case) If $\eps = (6,6)$, then $p \equiv 3 \Mod{4}$, and writing $\Gen$ for a set of split prime ideals of $K$ generating $\Cl_K$,
	\begin{itemize}
		\renewcommand\labelitemi{--}
		\item
		if there is a $\fq \in \Gen$ at which $E$ has potentially multiplicative or potentially good ordinary reduction, then either $p = \Nm(\fq)$, or $p$ divides one of the non-zero integers $A((6,6), \fq)$, $B((6,6), \fq)$ or $C_o((6,6), \fq)$;
		\item
		else $E$ has potentially good supersingular reduction at all $\fq \in \Gen$, and for all $q < p/4$ such that $q^2 + q + 1 \not\equiv 0 \Mod{p}$, if $q$ splits or ramifies in $K$, then $q$ does not split in $\Q(\sqrt{-p})$. Moreover, assuming GRH, this property of $p$ implies that
		\[ p \leq (16\log p + 16\log(12\Delta_K) + 26)^4. \]
	\end{itemize}
\end{enumerate}
\end{proposition}

\begin{proof}
We collect the previous results required to prove this Proposition. Item (1) is \Cref{prop:david_isog_char} and \Cref{lem:momose_lemma_1}. The first part of item (2) is \Cref{prop:ABC_div} and \Cref{prop:david_prop_215}, with the second part being \Cref{prop:ABC_div} and \Cref{cor:momose_weak_sharp}. Item (3) is \Cref{prop:david_prop_215} and the proof of \Cref{prop:type_1_primes}. For item (4), that $p \equiv 3 \mod{4}$ comes from item (5) of \Cref{prop:david_isog_char}; the first part of item (4) is then \Cref{cor:type_2_not_momose}, with the second part being \hyperlink{conditionCC}{Condition CC} and \Cref{prop:type_2_bound}.
\end{proof}

From this result one obtains the algorithm which our code implements as follows; this may be considered the full and explicit version of \Cref{alg:main} from the Introduction.

\begin{theorem}\label{thm:main_full}
Let $K$ be a quadratic field which is not imaginary quadratic of class number one. Make the following definitions:
{\small
\begin{align*}
\Delta_K, h_K, \Cl_K  &= \mbox{discriminant, class number and class group of } K\\
\Aux &= \mbox{\parbox{14cm}{finite set of split prime ideals of $K$ which, if $K$ is\\imaginary, are further required to be non-principal}}\\[1ex]
\AuxGen &= \mbox{\parbox{14cm}{finite set of ideal generating sets for $\Cl_K$, all ideals\\ being required to split in $K$}}\\[1ex]
S &= \left\{(0,12), (0,4), (0,8), (4,4), (4,8), (4,12), (4,6), (0,6)\right\}\\[1ex]
\eps_6 &= (6,6)\\[1ex]
ABC(\eps, \fq) &= \mbox{\parbox{14cm}{$\lcm(A(\eps, \fq), B(\eps, \fq), C(\eps, \fq), \Nm(\fq))$, for the integers\\$A$, $B$, $C$ defined in \Cref{prop:main_detailed}}}\\[1ex]
ABC_o(\Gen) &= \mbox{\parbox{14cm}{$\underset{\fq \in \Gen}\lcm\left(A(\eps_6, \fq), B(\eps_6, \fq), C_o(\eps_6, \fq), \Nm(\fq)\right)$, for the\\integer $C_o$ defined in \Cref{prop:main_detailed}}}\\[1ex]
\NotTypeOneTwoPrimes(K) &= \bigcup_{\eps \in S} \ \Supp \left( \gcd_{\fq \in \Aux}(ABC(\eps, \fq))\right)\\[1ex]
\Aux_{\geq 7} &= \mbox{\parbox{14cm}{finite set of split prime ideals of $K$ of characteristic $\geq 7$}}\\[1ex]
h_\fq &= \mbox{order of } \fq \mbox{ in }\Cl_K\mbox{, for $\fq$ a prime ideal of $K$}\\
D(\fq) &= \lcm\left(\Nm(\fq), \Nm(\fq)^{12h_\fq} - 1, C((0,0), \fq)\right)\\
\TypeOnePrimes(K) &= 
	\PrimesUpTo(61) \cup \left\{71 \right\} \cup \Supp\left( \gcd_{\fq \in \Aux_{\geq 7}} (D(\fq))\right)\\
\TypeTwoNotMomosePrimes(K) &= \Supp \left(\gcd_{\Gen \in \AuxGen}ABC_o(\Gen)\right)\\[2ex]
\mbox{ConditionCC}(K,p) &= \mbox{\parbox{14cm}{\textup{True} if for all $q < p/4$ such that $q^2 + q + 1 \not\equiv 0 \Mod{p}$ and \\$q$ splits or ramifies in $K$, it follows that $q$ does not split in\\$\Q(\sqrt{-p})$; \textup{False} otherwise}}\\[1ex]
\TypeTwoMomosePrimes(K) &= \left\{p : \mbox{ConditionCC}(K,p) \mbox{ is True} \right\}\\[1ex]
\TypeTwoPrimes(K) &= \TypeTwoNotMomosePrimes(K) \cup \TypeTwoMomosePrimes(K).
\end{align*}
}
Then
\begin{equation}\label{eqn:all_supersets}
\begin{aligned}
    \IsogPrimeDeg(K) \subseteq  \NotTypeOneTwoPrimes(K) &\cup \TypeOnePrimes(K) \\
        \cup \ \TypeTwoPrimes(K) &\cup \Supp(\Delta_K).
\end{aligned}
\end{equation}
Further, if $p \in \TypeTwoMomosePrimes(K)$, then assuming GRH, 
\[ p \leq (16\log p + 16\log(12\Delta_K) + 26)^4;\]
in particular $\TypeTwoMomosePrimes(K)$ and hence the right hand side of \Cref{eqn:all_supersets} is finite and explicitly computable.
\end{theorem}

Note that this strategy of computing a superset for $\IsogPrimeDeg(K)$ is used by Mazur in \cite{mazur1978rational}; compare, for example, \Cref{eqn:all_supersets} above with the very last sentence of his paper, just before Goldfeld's Appendix: ``it is clear that Proposition 8.1 is proved where $\mathcal{N}(K)$ is the finite set of primes $\mathcal{N}_1(K) \cup \mathcal{N}_2(K) \cup \mathcal{N}_3(K) \cup \mathcal{N}_4(K)$''.

\begin{remark}
In order to have some testing of the code, we require examples of isogeny primes larger than $71$ over specific quadratic fields. From the survey article \cite{gonzalez2004arithmetic}, together with one example from Box's Section 4.7 \cite{box2021quadratic}, we have the six examples shown in \Cref{tab:tests}.

\begin{table}[htp]
\begin{center}
\begin{tabular}{|c|c|c|}
\hline
Prime & Is Isogeny Prime over & Found by\\
\hline
$73$ & $\Q(\sqrt{-127})$ & Galbraith \cite{galbraith1999}\\
$73$ & $\Q(\sqrt{-31})$ & Box \cite{box2021quadratic}\\
$103$ & $\Q(\sqrt{5 \cdot 577})$ & Galbraith (\emph{loc. cit.})\\
$137$ & $\Q(\sqrt{-31159})$ & Galbraith (\emph{loc. cit.})\\
$191$ & $\Q(\sqrt{61 \cdot 229 \cdot 145757})$ & Elkies \cite{elkies1998}\\
$311$ & $\Q(\sqrt{11 \cdot 17 \cdot 9011 \cdot 23629})$ & Galbraith (\emph{loc. cit.})\\
\hline
\end{tabular}
\vspace{0.3cm}
\caption{\label{tab:tests}Instances of large isogeny primes in the literature which are rational over quadratic fields not imaginary quadratic of class number one.}
\end{center}
\end{table}

These examples serve as the basis of a unit testing framework, which has been implemented in \verb|test_quadratic_isogeny_primes.py|. For each of the above five quadratic fields, we check that the prime in the left column, as well as all primes in $\IsogPrimeDeg(\Q)$, are in output of the program. Moreover, the inclusion $\IsogPrimeDeg(\Q) \subseteq \IsogPrimeDeg(K)$ is checked for each quadratic field $K = \Q(\sqrt{D})$ for $|D| \leq 100$. Running all of these tests takes about three minutes on an old laptop.
\end{remark}

We end this section by stating the David-Larson-Momose-Vaintrob bound for a quadratic field $K$.

Recall from the discussion at the end of \Cref{sec:prelims} that David gave the bound $C(K,2(\Delta_K)^{Ah_K})$ for the primes $p$ which are not of Momose Type 1, 2 or 3. However, since we are in any case assuming GRH, we may replace the $2(\Delta_K)^{Ah_K}$ with $(4\log|\Delta_K|^{h_K} + 5h_K + 5)^2$, using Theorem 5.1 of \cite{bach1996explicit} (applied with $E = H_K$, the Hilbert class field of $K$).

As explained in \Cref{sec:type_1_primes}, $(1 + 3^{12h_K})^2$ gives a bound on the Momose Type 1 primes. A bound on the Momose Type 2 primes is given by the largest root $T_K$ of the equation
\[ f(x) = x - (16\log(x) + 16\log(12\Delta_K) + 26)^4.\]
Elementary calculus shows that $f$ is negative and decreasing at $x = 1$, admits a unique turning point at some $x > 1$, after which $f$ is strictly increasing; therefore $T_K$ is the unique root of $f$ in the range $(1,+\infty)$.
Finally our assumptions allow us to ignore Momose Type 3 primes. We thus make the following definition.

\begin{definition}
Let $K$ be a quadratic field which is not imaginary quadratic of class number one. Write $h_K$ for the class number of $K$ and $\Delta_K$ for the discriminant. We then define the \textbf{David-Larson-Momose-Vaintrob bound} for $K$ to be
\[ \DLMV(K) := \max((1 + 3^{12h_K})^2, T_K, C(K,(4\log|\Delta_K|^{h_K} + 5h_K + 5)^2)),\]
where the function $C(K,n)$ is given in \Cref{eqn:david_constant}.
\end{definition}

This gives us the following result.

\begin{proposition}
Let $K$ be a quadratic field which is not imaginary quadratic of class number one. Then, assuming GRH, $\DLMV(K)$ is a bound for the isogeny primes for $K$. That is, if there exists an elliptic curve over $K$ admitting a $K$-rational $p$-isogeny for $p$ prime, then $p \leq \DLMV(K)$.
\end{proposition}

\section{Weeding out the Pretenders}\label{sec:weeding}

Having now computed a superset for $\IsogPrimeDeg(K)$, one is reduced to an elaborate game of number-theoretic \emph{Whac-a-mole}. This section relies heavily on three different papers (\cite{bruin2015hyperelliptic}, \cite{ozman2019quadratic}, \cite{box2021quadratic}) on the subject of quadratic points on low genus modular curves; as such we begin by stating their combined result.

To do this, we define the dichotomy of \textbf{exceptional} and \textbf{non-exceptional} quadratic points on a general smooth curve $\mathcal{C}$ of genus at least $2$ defined over $\Q$. If $\mathcal{C}$ admits a degree $2$ map defined over $\Q$ to either $\PP^1$ (i.e. $\mathcal{C}$ is hyperelliptic) or to a rational elliptic curve $E$ of positive rank (in particular, $\mathcal{C}$ is bielliptic), then pulling back the rational points along the map will result in infinitely many quadratic points on $\mathcal{C}$. Since this is a clear source of infinitely many quadratic points, we call these non-exceptional. It is a theorem initially of Harris and Silverman (Corollary 3 in \cite{harris1991bielliptic}), often attributed to Abramovich and Harris (who actually extended this to higher degree points in Theorem 1 in \cite{abramovich1991abelian}) that these are the only two potential sources of infinitely many quadratic points. In particular, there are only finitely many quadratic points on $\mathcal{C}$ which \emph{do not} arise from pulling back rational points on $\PP^1$ or $E$ as above; we call these quadratic points \textbf{exceptional}.

We specialise now to $\mathcal{C} = X_0(N)$ being a modular curve. Ogg \cite{ogg1974hyperelliptic} determined the values of $N$ ($19$ of them) for which $X_0(N)$ is hyperelliptic, and Bars \cite{bars1999bielliptic} determined the values of $N$ ($10$ of them) for which $X_0(N)$ is bielliptic of positive rank (the infamous $N = 37$ falling into both categories). There are therefore $28$ values of $N$ for which $X_0(N)$ admits non-exceptional quadratic points, written explicitly in the literature as Theorem 4.3 in \cite{bars1999bielliptic}: $22$, $23$, $26$, $28$, $29$, $30$, $31$, $33$, $35$, $37$, $39$, $40$, $41$, $43$, $46$, $47$, $48$, $50$, $53$, $59$, $61$, $65$, $71$, $79$, $83$, $89$, $101$, $131$.

We may then summarise the main results of (\cite{bruin2015hyperelliptic}, \cite{ozman2019quadratic}, \cite{box2021quadratic}) as follows.

\begin{theorem}[Bruin-Najman, \"{O}zman-Siksek, Box]
Let $X_0(N)$ be a modular curve of genus $2$, $3$, $4$ or $5$, but excluding $N = 37$.
\begin{enumerate}
	\item
	If $X_0(N)$ admits non-exceptional quadratic points, then these points correspond (in the moduli interpretation) to elliptic $\Q$-curves.
	\item
	The finitely many exceptional quadratic points on $X_0(N)$ are listed in the tables \cite{bruin2015hyperelliptic}, \cite{ozman2019quadratic}, \cite{box2021quadratic}.
\end{enumerate}
For $N = 37$, the description of quadratic points is given in Section 5 of \cite{box2021quadratic}.
\end{theorem}

This combined result reduces our task to determining whether or not certain modular curves $X_0(N)$ admit non-exceptional $K$-points, for a given quadratic field $K = \Q(\sqrt{d})$. Since these correspond to elliptic $\Q$-curves of degree $p$ defined over $K$, they yield $\Q$-points on the quadratic twist $X^d(N)$ of $X_0(N)$ by the quadratic field $K$. For more details on this process see the introduction to \"{O}zman's paper \cite{ozman2012points} or Ellenberg's survey article \cite{ellenberg2004qcurves}. This reduces the problem to determining whether or not the $d$-twisted modular curve $X^d(N)$ admits a $\Q$-rational point.

If the Mordell-Weil rank of the Jacobian of this curve is less than the genus, then one can in principle apply Chabauty's method to determine the $\Q$-points, and in particular to determine whether or not they exist. If the rank of the Jacobian is zero, then this process can be automated in Magma, using the function \path{Chabauty0} due to Maarten Derickx and Solomon Vishkautsan. This happens in a few of the cases to be resolved later, so we give this method the name of \textbf{Twist-Chabauty0}.

Another way to rule out $\Q$-points on $X^d(N)$ is to show that there is a prime $p$ for which $X^d(N)(\Q_p)$ is empty; that is, to exhibit a \textbf{local obstruction} at a given place. This is the motivating question of \"{O}zman's paper \cite{ozman2012points}. We collect the results from this paper which we use in the sequel.

We thus adopt the notation of Section~4 of \emph{loc. cit.}. Let $\mathbb{K}$ denote the quadratic field $\Q(\sqrt{d})$, $N$ a square-free integer, and $p$ a prime that ramifies in $\mathbb{K}$ but not in $\Q(\sqrt{-N})$. Let $\nu$ be the prime of $\mathbb{K}$ lying over $p$, and $R$ the ring of integers of the completion $\mathbb{K}_\nu$. Denote by $\mathcal{X}_0(N)$ the Deligne-Rapoport model of $X_0(N)$, which is smooth and regular over $\ZZ[1/N]$. Letting $w_N$ be the Atkin-Lehner involution, \"{O}zman defines the set $S_N$ to be the primes $p$ for which there is a $w_N$-fixed, $\F_p$-rational point on the special fiber of $\mathcal{X}_0(N)_{/R}$, and relates this to the local solubility of $X^d(N)$ as follows.

\begin{theorem}[\"{O}zman, Theorem 4.10 in \cite{ozman2012points}]\label{thm:ozman_sieve}
Let $p$ be a prime ramified in $\Q(\sqrt{d})$ and $N$ a square-free integer such that $p$ is unramified in $\Q(\sqrt{-N})$. Then $X^d(N)(\Q_p) \neq \emptyset$ if and only if $p$ is in the set $S_N$.
\end{theorem}

Detrmining whether a prime $p$ is in the set $S_N$ is achieved by applying Proposition 4.6 of \emph{loc. cit.}. Although this proposition is stated as only applying for odd primes, it also holds for $p = 2$, as proved by \"{O}zman in a separate argument appearing between Remark~4.8 and Theorem~4.10 in \emph{loc. cit.}. The function \path{oezman_sieve(p,N)} in \path{quadratic_isogeny_primes.py} in the repository has implemented the concrete criterion expressed in \"{O}zman's Proposition 4.6, and returns \path{True} if and only if $p \in S_N$. 

We frame this method based on \"{O}zman's results into the following.

\begin{proposition}[\"{O}zman sieve]\label{prop:oezman_sieve}
Let $N \neq 37$ be a square-free integer, and $p$ a prime which ramifies in the quadratic field $\Q(\sqrt{d})$ and is unramified in $\Q(\sqrt{-N})$. Then, if \verb|oezman_sieve(p,N)| returns False, the twisted modular curve $X^d(N)$ has no $\Q$-rational point, and hence $X_0(N)$ does not admit a non-exceptional $\Q(\sqrt{d})$-point.
\end{proposition}

\begin{remark}
\begin{enumerate}
	\item
The reason for excluding $N = 37$ here is that it admits two distinct sources of non-exceptional quadratic points, and we lack a moduli interpretation for the elliptic curves corresponding to one of these; see Section 5 of \cite{box2021quadratic} for more details. For our purposes, since we know that $37 \in \IsogPrimeDeg(K)$ for every number field $K$, this will not be an issue for us.
	\item
Although \Cref{thm:ozman_sieve} requires the prime $p$ to be unramified in $\Q(\sqrt{-N})$, the implication
\[ p \notin S_N \Longrightarrow X^d(N)(\Q_p) = \emptyset\]
still holds if $p$ ramifies in both $\Q(\sqrt{d})$ and $\Q(\sqrt{-N})$; this is Proposition 4.5 of \cite{ozman2012points}. However, the concrete interpretation of $S_N$ given in Proposition 4.6 (which is implemented in \path{oezman_sieve}) requires that $p$ be unramified in $\Q(\sqrt{-N})$, so one cannot apply \verb|oezman_sieve(p,N)| at primes $p$ which ramify in \emph{both} $\Q(\sqrt{d})$ and $\Q(\sqrt{-N})$. Note that a literal reading of Parts (5) and (6) of Theorem 1.1 of \cite{ozman2012points} does suggest that we can apply \verb|oezman_sieve(p,N)| at any $p$ which ramifies in $\Q(\sqrt{d})$, and the confusion arises because the $S_N$ in Part (5) refers to the concrete description given in Proposition 4.6, while the $S_N$ in Part (6) refers to the original abstract definition given in the discussion between Proposition 4.5 and Proposition 4.6. We wish to make future readers aware of this possible source of confusion.
\end{enumerate}
\end{remark}

We illustrate the techniques hitherto collected to determine $\IsogPrimeDeg(\Q(\sqrt{5}))$. The techniques will subsequently be applied in the other cases of $\Q(\sqrt{-10})$ and $\Q(\sqrt{7})$, thereby proving \Cref{thm:main}. A fuller discussion of the implementation of the code for this section may be found in Section 8 of the user's guide of \emph{Quadratic Isogeny Primes}.

\begin{example}[$K = \Q(\sqrt{5})$]
Running the algorithm with \verb|--aux_prime_count 25| yields a superset as follows:
\[ \IsogPrimeDeg(\Q(\sqrt{5})) \subseteq \left\{\mbox{primes} \leq 71\right\} \cup \left\{73, 79, 163 \right\}. \]
We know that $\IsogPrimeDeg(\Q) \subseteq \IsogPrimeDeg(\Q(\sqrt{5}))$; so our task is to decide, for each $p$ in the following list, whether $X_0(p)(K)$ has any rational points beyond the two cusps defined over $\Q$:
\[\left\{23, 29, 31, 41, 47, 53, 59, 61, 71, 73, 79 \right\}.\]
\subsection*{Primes 23 and 47}
A search for $K$-rational points on $X_0(23)$ and $X_0(47)$ reveals that these modular curves do indeed admit $K$-points which are not $\Q$-points, showing that $23$ and $47$ are inside the set $\IsogPrimeDeg(K)$. Magma may also be used to compute the $j$-invariants of the elliptic curves supporting these prime-degree isogenies.

\subsection*{Primes 29 and 31}
These are both genus $2$ hyperelliptic curves, and as such their quadratic points were determined by Bruin and Najman \cite{bruin2015hyperelliptic}. We find (Tables 5 and 7 of \emph{loc. cit.}) that there are no exceptional points defined over $K$. To rule out the existence of non-exceptional points we apply the Twist-Chabauty0 method, since the ranks of the Jacobians of the $5$-twists of both of these modular curves can be determined in Magma to be zero.

\begin{remark}
The \"{O}zman sieve may also be used to rule out the non-exceptional $K$-points here; indeed in Example 4.11 of \cite{ozman2012points} the author shows that $X^5(29)(\Q_5) = \emptyset$.
\end{remark}

\subsection*{Prime 41}
Table 12 of \cite{bruin2015hyperelliptic} shows that this modular curve does not have an exceptional $K$-point. To rule out the existence of a non-exceptional $K$-point we run \verb|oezman_sieve(5,41)| which yields \verb|False| and hence we can conclude by \Cref{prop:oezman_sieve}.

\subsection*{Prime 53}
This is the first prime that is not in the list of Bruin and Najman. Fortunately Box \cite{box2021quadratic} has determined (Section 4.2 of \emph{loc. cit.}) that all quadratic points on $X_0(53)$ are non-exceptional. We may thus apply the \"{O}zman sieve to show that $53 \notin \IsogPrimeDeg(K)$. 

Note that this case also follows from Theorem 2.5 in \cite{gonzalez2004arithmetic}, since $53 \equiv 1 \Mod{4}$, but is not a norm of $K$.

\subsection*{Primes 59, 61, 71}
Table 17 of \cite{bruin2015hyperelliptic}, Section 4.3 of \cite{box2021quadratic}, and Table 18 of \cite{bruin2015hyperelliptic} respectively show that $X_0(59)$, $X_0(61)$, and $X_0(71)$ do not contain any exceptional quadratic points. The \"{O}zman sieve then shows that none of these three primes are in $\IsogPrimeDeg(K)$.

\subsection*{Prime 73}
Section 4.7 of \cite{box2021quadratic} determines all 11 quadratic points on $X_0(73)$ (besides the cusps), and none are defined over $K$. Note that this result relies on the determination of the $\Q$-points on $X_0^+(73)$, completed by Balakrishnan, Best, Bianchi, Lawrence, M\"{u}ller, Triantafillou and Vonk (Theorem 6.3 in \cite{balakrishnan2019two}).

\subsection*{Prime 79}
This is dealt with in the appendix to this paper, written jointly with Derickx.
\end{example}

Using the methods described in the above example, we may now prove the remaining two cases of \Cref{thm:main} more succinctly.

\begin{proof}[Proof of Theorem \ref{thm:main}]
\Cref{tab:main} collects the relevant information for showing that the primes in the set
\[\left\{23, 29, 31, 41, 47, 53, 59, 61, 71, 73 \right\}\]
do not arise as isogeny primes for $\Q(\sqrt{7})$ or $\Q(\sqrt{-10})$. The relevant result from the tables of Bruin-Najman and Box are cited, which show that, for each $p$ as above, if $p$ is an isogeny prime, then it must arise as a non-exceptional $K$-point; equivalently, arises as a quadratic $\Q$-curve. We also take this opportunity to summarise the case of $\Q(\sqrt{5})$, allowing all three instances to be summarised in one table. The entries in the table are then one of the following:

\begin{enumerate}
	\item
	TC0 - referring to the Twist-Chabauty0 method which deals with that case;
	\item
	the value \verb|q| at which \verb|oezman_sieve(q,p)| yields \verb|False|, showing that no such quadratic $\Q$-curve exists;
	\item
	`A.2' - referring to applying the method in the Appendix, given by \Cref{lemma:main}. This only applies to one case, and the Magma code to verify it may be found in \path{magma_scripts/appendix.m};
	\item
	\ding{55} - referring to the cited reference allowing us to conclude negatively immediately;
	\item
	$\checkmark$ - referring to that prime being an isogeny prime for that quadratic field;
	\item
	--- -  referring to that case not being relevant (this only applies to the prime $79$, where for $\Q(\sqrt{-10})$ and $\Q(\sqrt{7})$ we know that $79$ is not a possible isogeny prime).
\end{enumerate}

Running \verb|sage quadratic_isogeny_primes D --aux_prime_count 25|, for $D = 7$ and $-10$ then shows that $163$ is the only isogeny prime greater than $73$, and the proof of \Cref{thm:main} is complete.
\end{proof}

\begin{table}[htp]
\begin{center}
\begin{tabular}{|c|c|c|c|c|}
\hline
$p$ & Reference & $\Q(\sqrt{-10})$ & $\Q(\sqrt{5})$ & $\Q(\sqrt{7})$\\
\hline
$23$ &  Table 2 of Bruin-Najman & $2$ & \checkmark & 2\\
$29$ & Table 5 of Bruin-Najman & TC0 & TC0 & TC0\\
$31$ & Table 7 of Bruin-Najman & TC0 & TC0 & TC0\\
$41$ & Table 12 of Bruin-Najman & $5$ & $5$ & $7$\\
$47$ & Table 14 of Bruin-Najman & $2$ & $\checkmark$ & $7$\\
$53$ & Section 4.2 of Box, \Cref{appendix}  & $5$ & $5$ & A.2\\
$59$ & Table 17 of Bruin-Najman & $5$ & $5$ & $7$\\
$61$ & Section 4.3 of Box & $5$ & $5$ & $7$\\
$71$ & Table 18 of Bruin-Najman & $5$ & $5$ & $2$\\
$73$ & Section 4.7 of Box & \ding{55} & \ding{55} & \ding{55}\\
$79$ & \Cref{appendix} & --- & \ding{55} & ---\\
\hline
\end{tabular}
\vspace{0.3cm}
\caption{\label{tab:main}Deciding on possible isogeny primes. The numbers in the table correspond to primes where the \"{O}zman sieve was applied. ``TC0'' refers to an application of the Twist-Chabauty0 method, ``A.2'' refers to an application of \Cref{lemma:main}, \ding{55} means the reference cited is enough to decide negatively on that prime, $\checkmark$ means the prime is an isogeny prime for that field, and ``---'' means the corresponding case is not relevant. Bruin-Najman refers to \cite{bruin2015hyperelliptic} and Box to \cite{box2021quadratic}. In all cases apart from where $\checkmark$ is present, this table is ruling out primes.}
\end{center}
\end{table}
 
\begin{appendices}

\section{On $X_0(79)(\Q(\sqrt{5}))$}\label{appendix}
\begin{center}by Barinder S. Banwait and Maarten Derickx\end{center}
\medskip

In this appendix we establish the following result.

\begin{proposition}\label{prop:main}
The set of $\Q(\sqrt{5})$-rational points on the modular curve $X_0(79)$ consists only of the two $\Q$-rational cusps.
\end{proposition}

We begin with setting up notation and definitions.
\begin{align*}
p &= \mbox{prime};\\
J_0(p) &= \mbox{Jacobian variety of } X_0(p);\\
w_p &= \mbox{Atkin-Lehner involution on } X_0(p);\\
J_0(p)_+ &= \mbox{the sub-abelian variety }(1 + w_p)\cdot J_0(p);\\
J_0(p)_- &= \mbox{the sub-abelian variety }(1 - w_p)\cdot J_0(p).
\end{align*}
These latter subvarieties are referred to as the $+$ and $-$ parts of $J_0(p)$ respectively, and give a decomposition of $J_0(p)$ on which $w_p$ acts as $+1$, respectively~$-1$. See Chapter 2, Section 10 of \cite{Mazur3} for more details.

Recall that the \textbf{gonality} of a geometrically integral curve $X$ over a field $K$ is the smallest possible degree of a dominant rational map $X \to \PP^1_K$. Throughout this appendix we will be concerned solely with the case that the gonality of a curve is strictly greater than~$2$; equivalently, we assume throughout that $X$ is neither hyperelliptic, nor of genus~$0$ or $1$.

\Cref{prop:main} will follow as a corollary of the following lemma.

\begin{lemma}\label{lemma:main}
Let $K$ be a number field, and $p$ a prime such that
\begin{enumerate}
    \item $J_0(p)_-(\Q) = J_0(p)_-(K)$;
    \item the gonality of $X_0(p)$ is strictly greater than~$2$.
\end{enumerate}

Then for all $x \in X_0(p)(K)$, $x$ satisfies one of the following:
\begin{dichotomy}
\item $x \in X_0(p)(\Q)$;
\item $x$ is fixed by $w_p$; in particular the elliptic curve corresponding to $x$ has complex multiplication by an order of $\Q(\sqrt{-p})$.
\end{dichotomy}
\end{lemma}
Before proving the lemma we first verify that the conditions (1) and (2) are satisfied in our situation for $K = \Q(\sqrt{5})$ and $p = 79$.

\begin{lemma}\label{lem:no_extra_points}
$J_0(79)_-(\Q(\sqrt{5})) = J_0(79)_-(\Q) \cong \ZZ/13\ZZ$.
\end{lemma}

\begin{proof}
By considering Galois orbits of newforms of weight~$2$ and level~$\Gamma_0(79)$ (which can be readily done using the \href{https://www.lmfdb.org/ModularForm/GL2/Q/holomorphic/?hst=List&level=79&weight=2&char_order=1&search_type=List}{\emph{$L$-functions and Modular Forms Database}} \cite{LMFDB}), we find that $J_0(79)_+$ is the elliptic curve with Cremona label 79a1 \cite{cremona1997algorithms}, and $J_0(79)_-$ is the five-dimensional modular abelian variety associated to the Galois orbit of weight~$2$ newform with LMFDB label \href{https://www.lmfdb.org/ModularForm/GL2/Q/holomorphic/79/2/a/b/}{79.2.a.b}, which is simple over $\Q$.

Both $J_0(79)_-$ and its $\Q(\sqrt{5})$-twist (corresponding to the newform with LMFDB label \href{https://www.lmfdb.org/ModularForm/GL2/Q/holomorphic/1975/2/a/k/}{1975.2.a.k}) have rank~$0$ over $\Q$, so $J_0(79)_-$ has rank~$0$ over $\Q(\sqrt{5})$. Additionally, from Theorem~1 and Theorem~3 of \cite{Mazur3} we know that the torsion subgroup of $J_0(79)_-(\Q)$ is isomorphic to $\ZZ/13\ZZ$.

It remains to show that the torsion subgroup of $J_0(79)_-(\Q(\sqrt{5}))$ is isomorphic to $\ZZ/13\ZZ$ as well. This can be readily checked in Magma (see the file \path{magma_scripts/appendix.m} in \cite{isogeny_primes}) by computing the number of $\F_{p^2}$-rational points on $J_0(79)$ for $p=2,3,5$. The GCD of the resulting values is 13, so the torsion subgroup of $$J_0(79)_-(\Q(\sqrt{5})) \subseteq J_0(79)(\Q(\sqrt{5}))$$ is at most of order 13. But since $J_0(79)_-(\Q)_{tors} \cong \ZZ/13\ZZ$, it has order exactly $13$.
\end{proof}
\begin{remark}
The above proof actually shows that the torsion of $J_0(79)$ does not grow when considered over an arbitrary quadratic field, so that for quadratic fields $K$ one has $J_0(79)_-(\Q) = J_0(79)_-(K)$ if and only if the quadratic twist of $J_0(79)_-$ by $K$ has rank 0.
\end{remark}

\begin{proposition}
$X_0(79)$ has gonality $4$.
\end{proposition}
\begin{proof}
Since the genus of $X_0(79)$ is 6 it is neither elliptic or rational. That it is not hyperelliptic is Theorem 2 of \cite{ogg1974hyperelliptic}. Thus the gonality is at least $3$.

On the other hand, the gonality is at most $4$, since $X_0^+(79) := X_0(79)/w_{79}$ is an elliptic curve (indeed, the same elliptic curve 79a1 identified in the proof of \Cref{lem:no_extra_points}).

That the gonality is not $3$ follows from Theorem 2.1 (e) of \cite{JEON2021344}.
\end{proof}

Having established that the conditions of \Cref{lemma:main} are satisfied in our case of interest, it is straightforward to show that \Cref{prop:main} follows from \Cref{lemma:main}.

\begin{proof}[Proof of \Cref{prop:main}]
Let $x \in X_0(79)(\Q(\sqrt{5}))$ be a point. From \Cref{lemma:main} we have that either $x \in X_0(79)(\Q)$ or the elliptic curve corresponding to $x$ has complex multiplication by $\Q(\sqrt{-79})$. Since $\Q(\sqrt{-79})$ has class number 5 we know that none of the fixed points of $w_p$ are defined over a quadratic field, let alone over $\Q(\sqrt{5})$; thus we must have $x \in X_0(79)(\Q)$. That $x$ must therefore be one of the two $\Q$-rational cusps follows from Theorem~1 of \cite{mazur1978rational}.
\end{proof}

Before proving \Cref{lemma:main}, we require one final lemma.

\begin{lemma}
Let $p$ be a prime such that the gonality of $X_0(p)$ is strictly greater than $2$, and let $f$ be the map 
\begin{align*}
    f : X_0(p) & \to J_0(p) \\
             x & \mapsto w_{p}(x) - x.
\end{align*} 
Then the following statements hold.
\begin{enumerate}[label=(\roman*)]
    \item $f$ is injective away from the fibre above $0$;
    \item The points in the fibre above $0$ are fixed points of $w_p$ and correspond to elliptic curves with complex multiplication by $\Q(\sqrt{-p})$.
\end{enumerate}
\end{lemma}

\begin{proof}
\begin{enumerate}[label=(\roman*)]
    \item 
    
    Suppose $f(x)=f(y)$. Then $w_{p}(x)+y$ is linearly equivalent to $w_{p}(y)+x$ as degree~$2$ divisors. Since $X_0(p)$ has gonality strictly greater than $2$, we get equality of the divisors, not just up to linear equivalence. This means that either $x = y$; or $w_{p}(x) = x$ and $w_{p}(y) = y$, in which case $f(x) = f(y) = 0$.
    
    \item 
    
    Since $X_0(p)$ is not rational the fibre above 0 consists of the fixed points of $w_p$. It is well-known that these fixed points correspond to elliptic curves with complex multiplication by $\Q(\sqrt{-p})$; see for example Section~2 of \cite{ogg1974hyperelliptic}.
    
\end{enumerate}
\end{proof}

We are now finally ready for the proof of \Cref{lemma:main}.
\begin{proof}[Proof of \Cref{lemma:main}]
Let $x \in X_0(p)(K)$ be a point. If $f(x)=0$ then by the above lemma we are in case (II) of \Cref{lemma:main}, so from now on assume that $f(x) \neq 0$. The image of $f$ is contained in $J_0(p)_-$, so the first condition of \Cref{lemma:main} implies that $f(x) \in J_0(p)(\Q)$. But since $f$ is injective away from the fibre above $0$, we know that $x$ is the only element in its fibre, and thus must also be $\Q$-rational; i.e. we are in case (I) of \Cref{lemma:main}.
\end{proof}
\end{appendices}

\bibliographystyle{amsplain}
\bibliography{isogenyprimes.bbl}

\end{document}